\documentclass[12pt]{article}

\usepackage[latin1]{inputenc}
\usepackage[T1]{fontenc}
\usepackage[english]{babel}
\usepackage{amsmath,amssymb,amsthm}
\usepackage{mathrsfs} 
\usepackage{graphicx}
\usepackage{bbm}
\usepackage{color}
\usepackage{comment}
\usepackage[textsize=small]{todonotes}       
\usepackage{color}

\newtheorem{theorem}{Theorem}[section]

\newtheorem{definition}{Definition}[section]
\newtheorem{lemma}{Lemma}[section]
\newtheorem{corollary}{Corollary}[section]

\newtheorem{proposition}[theorem]{Proposition}

\newtheorem{remark}{Remark}[section]
\newtheorem{cond}{Condition}[section]

\newcommand{\be}{\begin{equation}}
\newcommand{\ee}{\end{equation}}
\newcommand{\nn}{\nonumber}

\newcommand{\prob}{\mathbb P}
\newcommand{\expec}{\mathbb E}

\newcommand{\atanh}{{\rm atanh}}

\newcommand{\sss}{\scriptscriptstyle}

\numberwithin{equation}{section}
\newcommand{\vep}{\varepsilon}
\newcommand{\e}{{\rm e}}
\newcommand{\gs}{\left(G_{\sss N}\right)_{N\geq1}}
\newcommand{\rem}[1]{}

\def\1{{\mathchoice {1\mskip-4mu\mathrm l}      
{1\mskip-4mu\mathrm l}
{1\mskip-4.5mu\mathrm l} {1\mskip-5mu\mathrm l}}}
\newcommand{\indic}[1]{\1_{\{#1\}}}

\newcommand{\q}{Q_{\sss N}} 
\newcommand{\m}{\mu_{G_{\sss N}}} 
\newcommand{\p}{P_{\sss N}} 
\newcommand{\CMNd}{\mathrm{CM}_{\sss N}(\boldsymbol{d})}
\newcommand{\CMNreg}{\mathrm{CM}_{\sss N}(\boldsymbol{r})}
\newcommand{\CMNtwo}{\mathrm{CM}_{\sss N}(\mathrm{\textbf{2}})}
\newcommand{\CMNonetwo}{\mathrm{CM}_{\sss N}(\textbf{1},\textbf{2})}

\newcommand{\col}[1]{\textcolor[rgb]{0,0,0}{#1}}

\newcommand{\eqn}[1]{\begin{equation}#1\end{equation}}
\newcommand{\eqan}[1]{\begin{align}#1\end{align}}

\newcommand{\convd}{\stackrel{\sss {\mathcal D}}{\longrightarrow}}

\newcommand{\N}{\mathbb{N}}
\newcommand{\R}{\mathbb{R}}
\newcommand{\sign}{\mathrm{sign}}

\newcommand{\old}[1]{}

\begin{document}

\title{Quenched central limit theorems\\
for the Ising model on random graphs}
\author{
Cristian Giardin\`a$^{\textup{{\tiny(a)}}}$,
Claudio Giberti$^{\textup{{\tiny(b)}}}$,\\
Remco van der Hofstad$^{\textup{{\tiny(c)}}}$,
Maria Luisa Prioriello$^{\textup{{\tiny(a,c)}}}$\;.
\\
{\small $^{\textup{(a)}}$
University of Modena and Reggio Emilia},
{\small via G. Campi 213/b, 41125 Modena, Italy}
\\
{\small $^{\textup{(b)}}$
University of Modena and Reggio Emilia},
{\small Via Amendola 2, 42122 Reggio Emilia, Italy}
\\
{\small $^{\textup{(c)}}$ 
Eindhoven University of Technology,}
{\small P.O. Box 513, 5600 MB Eindhoven, The Netherlands  }
}
\maketitle

\pagenumbering{arabic}

\begin{abstract}
\noindent
The main goal of the paper is to prove central limit theorems for
the magnetization rescaled by {$\sqrt{N}$} for the Ising model on 
random graphs with \col{$N$} vertices. Both random quenched and averaged quenched measures are considered.
We work in the uniqueness regime $\beta>\beta_c$ or $\beta>0$ and $B\neq0$,
where $\beta$ is the inverse temperature, $\beta_c$ is the 
critical inverse temperature and $B$ is the external magnetic field.
In the random quenched setting our results apply to general tree-like random graphs
(as introduced by Dembo, Montanari and further studied by Dommers and the first and 
third author) and our proof follows that of Ellis in $\mathbb{Z}^d$. 
For the averaged quenched setting, we specialize
to two particular random graph models, namely the 2-regular configuration model and
the configuration model with degrees 1 and 2. In these cases our proofs are based on explicit computations
relying on the solution of the one dimensional Ising models.

\end{abstract}


\section{Introduction and main results}
The study of asymptotic results (law of large numbers, central limit theorem, large deviations) is a 
subject of prominent  interest in probability theory.
Ising models on random graphs are ideal models of {\em dependent random variables} in the presence of {\em two sources of randomness}.
They posses a dependence structure between Ising random variables, given 
by the Boltzmann-Gibbs measure, and an extra level of (spatial) disorder given by
the random graph. This is the source of a distinction between the {\em random quenched 
measure}, the {\em averaged quenched measure} and the {\em annealed measure}. 
In the first two cases the environment given by the random graph is frozen, in the last 
case the random environment and the spin variables are treated on the same foot (see Definition 
\ref{measure} for a precise statement and Section \ref{diff_rq_aq} for further comments 
on different type of fluctuations). 

The goal of this paper is to derive asymptotic results for the sum of the spin variables 
in the Ising model on random graphs at the level of the {\em law of large numbers} (LLN) and 
of the {\em central limit theorem} (CLT) with respect to the random quenched and averaged quenched measures. To simplify the analysis we will restrict to the situation 
where the model has only one pure state.  The central limit theorems in the annealed setting for some of the models considered here will be given elsewhere \cite{GGPvdH2}.

There is a rising interest for the study of stochastic processes on random graphs,
mainly due to the link of such models with applications in fields as diverse
as social systems and combinatorial optimization problems.
In particular, a large class of models for which a fairly detailed picture
emerged is that of random graphs that are {\em locally tree-like},
meaning that the local structure around any vertex of the graph is given by a unimodular random tree.  
This class of random graphs arises in various applied contexts \cite{AB, New, NewBook}. Ising models on tree-like random graphs have been first studied with 
the tools of theoretical physics \cite{LVVZ,DGM,DGM2} and more recently they have been the subject 
of rigorous mathematical studies \cite{DG,DM,DGH,MMS,DMS,DMSS,DGH2}.
In particular, a full control has been achieved for the quenched free energy
and the main thermodynamical observables. We will leverage on the knowledge about the quenched free energy in the proof of the \col{{\em random quenched}} CLT for the (rescaled) total spin. It turns out that such a proof can be performed along the lines of  Ellis' proof  of the CLT  in the  $\mathbb{Z}^d$ case \cite{El}. The existence of the limiting cumulant generating function  of the total spin with respect to the  random quenched measure will be a direct consequence of the existence of the thermodynamic limit of the quenched free energy.

On the other hand, the local tree-like assumption does not allow to take
into account the fluctuations due to the spatial structure that become relevant in the study of the \col{{\em averaged quenched}} CLT for the (rescaled) total spin. 
There is therefore a need to expand the formalism to study Ising model
on random graph towards the situation in which the structure of the random
graph in a local neighborhood of a random vertex is allowed to be different
from a tree.

At present, there is no general formalism to treat such situations and therefore we analyze specific models, namely the configuration model with vertex degree 2 and the configuration model with vertex degrees 1 and 2,  which serve as benchmarks. Whereas it is well established that graph fluctuations do not play
any role in the {\em quenched} free energy in the general class of tree-like graphs (indeed averaged quenched and random quenched are the same, see Corollary \ref{corol_avquetermlim}), our analysis will show that fluctuations of the random graph play a crucial role in the study of the rescaled magnetization in the {\em averaged quenched }  set-up.
In particular, when averaging the random Boltzmann-Gibbs distribution with respect
to the randomness of the graph, the variance of the Gaussian limiting law
of the observables satisfying a central limit theorem in the thermodynamic 
limit (such as the rescaled magnetization) is in general affected by the graph 
fluctuations. 
 
The paper is organized as follows. We define the models, state and discuss our main results in this section,
including comments on why this analysis is needed and, so far, not covered in the existing formalism for statistical mechanics models on random graph. The rest of the paper is devoted to the proofs of the theorems stated in the present section.

In Section \ref{two}, in the general context  of locally tree-like random graphs, we prove our results with respect to the {\em random quenched measure}.  We establish the rate at which the law of large numbers for the empirical magnetization is reached, and we prove a  central limit theorem for the (centered)  empirical 
magnetization rescaled by the square root of the volume.  As in the case of the Ising model on the lattice or on the complete graph \cite{El}, the variance appearing in the CLT result is given by the spin susceptibility.

In Sections \ref{three} and \ref{four} we obtain asymptotic results with respect to the {\em averaged quenched measure}. While the law of large numbers can easily be formulated for the entire class of locally tree-like random graphs, the scaling to a normal random variable turns out to be much more challenging. Thus,  we restrict ourselves to the aforementioned  {\em configuration models} and consider the two simplest cases. In particular, we will consider  the configuration model with all vertices having degree two, i.e. 
the 2-regular random graph in  Section \ref{three}, and  the case where a fraction of the vertices  has degrees one and the remaining fraction has degree two, in Section \ref{four}.  In both cases we obtain the central limit theorem by an explicit computation. In the first case, we find that the CLT with respect to the averaged quenched measure is the same as  the CLT with respect to the random quenched measure. In particular 
the variance of the Gaussian law is given by the susceptibility of the Ising model in one dimension. In the second case, we prove that the asymptotic  variance of the rescaled empirical magnetization is {\em larger} than the susceptibility. This difference originates from the fluctuations of the connected component sizes 
of the  configuration model with degrees 1 and 2, that also follows a Gaussian law  in the limit of very large graphs. We argue this behavior to be general, i.e., we conjecture the existence of two different CLT theorems whenever the vertex degrees are not all the same. When all vertex degrees are equal, so that we are considering a random regular graph, we conjecture the variances of the random quenched and averaged quenched cases to be the same.

\vspace{.3cm}

\subsection{Locally tree-like random graphs}
\noindent
The class of models that we consider in this paper is the same as defined in \cite{DM,DGH}. Namely we study  random graph sequences $\left(G_{\sss N}\right)_{N\geq1}$ which are assumed to be \emph{locally like a homogeneous random tree, uniformly sparse} and whose \emph{degree distribution} has a  \emph{strongly finite mean}. 
In order to formally state these assumptions, we need to introduce some notation.

Let the integer-valued random variable $D$ have distribution $P = \left(p_k\right)_{k\geq1}$, i.e. $\mathbb{P}\left(D=k\right) = p_k$, for $k=1,2,\ldots $\; . We define the \emph{size-biased law} $\rho =\left(\rho_k\right)_{k \geq 0}$ of $D$  by
	\be\label{rho_kappa}
	\rho_k = \frac{\left(k+1\right)p_{k+1}}{\mathbb{E}\left[D\right]},
	\ee
where the expected value of $D$ is supposed to be finite, and let $\nu$ be the average value of $\rho$, i.e.,
	\be\label{nu_d}
	\nu := \sum_{k\geq0} k \rho_k = \frac{\mathbb{E}\left[D\left(D-1\right)\right]}
	{\mathbb{E}\left[D\right] }, 
	\ee
that will play an important role in what follows. 

The {\em random rooted tree} $\mathcal{T}\left(D, \rho, \ell\right)$ is a branching process with $\ell$ generations, where the root offspring has distribution $D$ and the vertices in each next generation have offspring that are \emph{independent and identically distributed} (i.i.d.) with distribution $\rho$. 

We write that an event $\mathcal{A}$ holds \emph{almost surely} (a.s.) if $\mathbb{P}\left(\mathcal{A}\right) =1$. The ball $B_i\left(r\right)$ of radius $r$ around vertex $i$ of a graph $G_{\sss N}$ is defined as the graph induced by the vertices at graph distance at most $r$ from vertex $i$. For two rooted trees $\mathcal{T}_1$ and $\mathcal{T}_2$, we write that $\mathcal{T}_1 \cong \mathcal{T}_2$, when there exists a bijective map from the vertices of $\mathcal{T}_1$ to those of $\mathcal{T}_2$ that preserves the adjacency relations.

Throughout this paper, $[N]:=\{1,\ldots, N\}$ will be used to denote the vertex set of $G_{\sss N}$.

\begin{definition}[Local convergence to homogeneous trees] \label{def_localconvtree}Let $\mathbb{P}_{\sss N}$ denote the law induced on the ball $B_i\left(t\right)$ in $G_{\sss N}$ centered at a uniformly chosen vertex $i \in [N] = \left\{1,...,N\right\}$. We say that the graph sequence $\left(G_{\sss N}\right)_{N\geq1}$ is \emph{locally tree-like} with asymptotic degree distributed as  $D$ when, for any rooted tree $\mathcal{T}$ with $t$ generations, a.s.,
	\be
	\lim_{N \rightarrow \infty} \mathbb{P}_{\sss N} \left(B_i\left(t\right) \cong \mathcal{T} \right) 
	= \mathbb{P}\left(\mathcal{T}\left(D, \rho, t\right) \cong \mathcal{T}\right).
	\ee
\end{definition}

\noindent
This property implies, in particular, that the degree of a uniformly chosen vertex in $G_{\sss N}$ is asymptotically distributed as $D$.

\begin{definition}[Uniform sparsity] 
\label{Uni_spar}
We say that the graph sequence $\left(G_{\sss N}\right)_{N\geq1}$ is \emph{uniformly sparse} when a.s.,
	\be
	\lim_{\ell \rightarrow \infty} \limsup_{N \rightarrow \infty} \frac{1}{N} 
	\sum_{i \in [N]} d_i \mathbbm{1}_{\{d_i \geq \ell\}}= 0,
	\ee
where $d_i$ is the degree of vertex $i$ in $G_{\sss N}$ and $\mathbbm{1}_{\mathcal{A}}$ denotes the indicator of the event $\mathcal{A}$.
\end{definition}

\noindent
A consequence of local convergence  and uniform sparsity is that the  number of edges per vertex converges a.s.\ to $\mathbb{E}[D]/2$ as $N$ goes to infinity.

\begin{definition}[Strongly finite mean degree distribution]
\label{Stro_fin}
We say that the degree distribution $P$ has \emph{strongly finite mean} when there exist constants $\tau > 2$ and $c>0$ such that
	\be
	\sum_{i=k}^{\infty} p_i \leq ck^{- \left(\tau -1\right)}.
	\ee
\end{definition}

\noindent
A specific example of a random graph model belonging to the class defined by these properties will be considered in detail in this paper: the configuration model under certain regularity conditions on its degree. We will introduce the configuration model in the next subsection.



\vspace{0.3cm}
\noindent
\subsection{The configuration model}\label{sectCM}
The configuration model is a multigraph, that is, a graph possibly having self-loops and multiple edges between pairs of vertices, with fixed degrees.
Fix an integer $N$ and consider a sequence of integers  $\boldsymbol{d}=\left(d_i\right)_{i \in \left[N\right]}$. The aim is to construct an undirected multigraph with $N$ vertices, where vertex $j$ has degree $d_j$. We assume that $d_j \geq 1$ for all $j \in [N]$ and we define the {\em total degree}
	\be
	\ell_{\sss N} \;:=\; \sum_{i \in [N]} d_{i}.
	\ee
We assume $\ell_{\sss N}$ to be even in order to be able to construct the graph. 

Assuming that initially $d_{j}$ half-edges are attached to each vertex $j\in [N]$,  one way of obtaining a multigraph with the given degree sequence is to pair the half-edges belonging to the different vertices in a uniform way. Two half-edges together form an edge, thus creating the edges in the graph. 
To construct the multigraph with degree sequence  $\boldsymbol{d}$,  the half-edges are numbered in an arbitrary order from 1 to $\ell_{\sss N}$. Then we start by randomly connecting the first half-edge with one of the $\ell_{\sss N} -1$ remaining half-edges. Once paired, two half-edges form a single edge of
the multigraph. We continue the procedure of randomly choosing and pairing the half-edges until all half-edges are connected, and call the resulting graph the \emph{configuration model with degree sequence $\boldsymbol{d}$}, abbreviated as $\CMNd$. Interestingly, $\CMNd$ conditioned on simplicity is a uniform random graph with degree sequence $\boldsymbol{d}$ (see e.g., \cite[Chapter 7]{vdH}).

\noindent
We will consider, in particular, the following models:
\begin{itemize}
\item[(1)] The \emph{2-regular random graph}, denoted by $\CMNtwo$, which is the configuration model with $d_i = 2$ for all $i \in \left[N\right]$.
\item[(2)] The configuration model with $d_i \in \left\{1,2\right\}$ for all $i \in \left[N\right]$, denoted by  $\CMNonetwo$, in which, for a given $p \in \left[0,1\right]$,  we have $N -\lfloor pN\rfloor$ vertices with degree 1 and $\lfloor pN \rfloor$ vertices with degree 2. 
\end{itemize}

\begin{remark}\label{alt_costr}
The  configuration model $\CMNonetwo$ can be implemented by assigning to each vertex degree 2 with probability $p$ or degree 1 with probability $1-p$, conditioned to having $\lfloor pN \rfloor$ vertices of degree 2. Another configuration model would be obtained by considering  the independent Bernoulli assignment  to each vertex. This  yields a random graph that  only on average has  $\lfloor pN \rfloor$ vertices of degree 2.  
\end{remark}

\noindent
The degree sequence of the configuration model  $\CMNd$  is often assumed to satisfy a  {\em regularity condition}, which is expressed as follows. Denoting the degree of an uniformly chosen vertex $V_{\sss N} \in [N]$ by $D_{\sss N} = d_{V_{\sss N}}$, we assume that the following property is satisfied:

\begin{cond}[Degree regularity] There exists a random variable $D$ with finite second moment such that, as $N\rightarrow\infty$, 
\label{cond-DR-CM}
\begin{itemize}
\item[$(a)$] $D_{\sss N} \stackrel{\cal D}{\longrightarrow} D$,
\item[$(b)$] $\mathbb{E}[D_{\sss N}] \rightarrow \mathbb{E}[D] < \infty$,
\item[$(c)$] $\mathbb{E}[D_{\sss N}^2] \rightarrow \mathbb{E}[D^2] < \infty$,
\end{itemize}
where  $\,\stackrel{\cal D}{\longrightarrow}\,$  denotes convergence in distribution.  Further, we assume that $\mathbb{P}(D\ge 1)=1$.
\end{cond}

\noindent
Referring to \cite{vdH}, we have that $\CMNd$ is a uniformly sparse locally tree-like graph when the properties $(a)$ and $(b)$ of the Condition \ref{cond-DR-CM} hold. Moreover when the property $(c)$ of the Condition \ref{cond-DR-CM} holds, we have also strongly finite mean. \\
It it easy to see that the degree sequences of $\CMNtwo$ and  $\CMNonetwo$ satisfy Condition \ref{cond-DR-CM}. For example, for $\CMNonetwo$, we have $\prob(D_{\sss N}=2)=1-\prob(D_{\sss N}=1)=\lfloor pN \rfloor/N\rightarrow p$, so that $\prob(D=2)=1-\prob(D=1)=p$.

\subsection{Measures and thermodynamic quantities}
\label{iniz_def}
We continue by introducing the Ising model.
Two probability measures are of interest.
We  define them on finite graphs with $N$ vertices
and then study asymptotic results  in the limit $N\to\infty$. \\We denote by  $G_{\sss N}=(V_{\sss N},E_{\sss N})$ a random graph with vertex set $V_{\sss N} = \left[N\right]$ and edge set $E_{\sss N} \subset V_{\sss N} \times V_{\sss N}$ and by $\mathcal G_{\sss N}$ the set of all possible graphs with $N$ vertices. 
For any $N\in {\mathbb N}$, we denote by  $\q$ the law of the graphs with $N$ vertices belonging to $\mathcal G_{\sss N}$.
\begin{definition}[Measures and expectations]
\label{measure}
For spin variables  $\sigma = \left (\sigma_1,\ldots,\sigma_{\sss N}\right )$
taking values on the space of spin configurations $\Omega_{\sss N}=\{-1,1\}^N$
we  consider the following measures:
\begin{description}
%
\item[(i) Random quenched measure.] 
For a given realization $G_{\sss N}\in \mathcal G_{\sss N}$ 
the random quenched measure coincides with the random Boltzmann--Gibbs distribution
	\begin{equation}
	\label{bg}
	\m(\sigma)
	=\frac{\exp \left[ \beta \sum_{(i,j)\in E_{\sss N}}{\sigma_i \sigma_j}+B \sum_{i \in[N]} {\sigma_i}\right]}
	{Z_{G_{\sss N}} \left( \beta, B \right)},
	\end{equation}
where
	\begin{equation}
	Z_{G_{\sss N}} \left( \beta, B \right)
	= \sum_{\sigma \in \Omega_{\sss N}} \exp \left[ \beta \sum_{(i,j)\in E_{\sss N}}{\sigma_i \sigma_j} 
	+ B \sum_{i \in[N]} {\sigma_i} \right] 
	\end{equation}
is the partition function. Here $\beta \ge 0$ is the inverse temperature and $B\in\mathbb{R}$ is the uniform external magnetic field.
\item[(ii) Averaged quenched measure.] 
This law is obtained by averaging the random Boltzmann--Gibbs distribution over all possible
random graphs, i.e.,
	\begin{equation}
	\p(\sigma) = \q(\m(\sigma))
	=\q\left(\frac{\exp \left[ \beta \sum_{(i,j)\in E_{\sss N}}{\sigma_i \sigma_j} 
	+ B \sum_{i \in[N]} {\sigma_i}  \right] }{Z_{G_{\sss N}} \left( \beta, B \right)}\right).
	\end{equation}
\end{description}
\end{definition}

\noindent
An extensive discussion about these two settings can be found in Section \ref{diff_rq_aq}.\\
\noindent
With a slight abuse of notation, in the following, we use the same symbol to denote both measures and their corresponding \col{expectations}. 
Moreover, we remark that all the measures defined above depend sensitively on the two parameters $(\beta,B)$. However, for the sake of notation, we drop the dependence of the measures on these parameters from the notation.  We denote the partition function by $Z_{\sss N}$ instead $Z_{G_{\sss N}}$ and sometimes we use  $\text{Var}_{\mu}(X)$ to denote the variance of a random variable $X$ with law $\mu$.

\vspace{0.4cm}

\begin{definition}[Thermodynamic quantities \cite{DM,DGH}]
\label{def_th_quant}
For a given $N\in\mathbb{N}$, we introduce the following thermodynamics quantities
at finite volume:
\begin{itemize}

\item[(i)] The  \emph{random quenched pressure}:
	\be\label{press_N}
	\psi_{\sss N} (\beta, B) \,=\, \frac{1}{N} \log Z_{\sss N} \left( \beta, B \right)  \;.
	\ee

\item[(ii)] The  \emph{averaged quenched pressure}:
	\be\label{av_qu_press_N}
	\overline{\psi}_{\sss N} (\beta, B) \,=\, \frac{1}{N} \q(\log Z_{\sss N} \left( \beta, B \right))  \;.
	\ee

\item[(iii)] The  \emph{random quenched} and \emph{averaged quenched magnetizations},  respectively :
	\be\label{magn_N}
	M_{\sss N} (\beta, B) \,=\, \m\left(\frac{S_{\sss N}}{N}\right), \quad \quad 
	\overline{M}_{\sss N} (\beta, B) \,=\, \p\left(\frac{S_{\sss N}}{N}\right),
	\ee
where the  \emph{total spin} is defined as
	\be
	S_{\sss N}= \sum_{i \in [N]} \sigma_i \;.
	\ee


\item[(iv)] The \emph{random quenched susceptibility}:
	\be\label{susc_N}
	\chi_{\sss N}(\beta,B) \,=\, \text{Var}_{\m} \left(\frac{S_{\sss N}}{\sqrt{N}}\right) \, 
	=\, \frac{\partial}{\partial B} M_{\sss N} (\beta, B). 
	\ee
We also define the variance with respect to the average quenched measure as
	\be\label{var_risp_aq}
	\quad \quad \overline{\chi}_{\sss N}(\beta,B) \,=\, \text{Var}_{\p} \left(\frac{S_{\sss N}}{\sqrt{N}}\right). 
	\ee

\end{itemize}
\end{definition}

\noindent
We are interested in the thermodynamic limit of these quantities, i.e., their limits as $N\to \infty$. In this limit critical phenomena  may appear.  When ${\cal M}(\beta, B):=\lim_{N\to \infty} {\cal M}_{\sss N}(\beta, B)$, where $ {\cal M}_{\sss N}(\beta, B)$ is any of the magnetizations defined in (\ref{magn_N}), criticality manifests itself in the behavior of the {\em spontaneous magnetization}  defined as ${\cal M}(\beta, 0^{+})=\lim_{B \downarrow 0} {\cal M}(\beta, B)$. In fact, the {\em critical inverse temperature}  is defined as
	\be
	\beta_{c}:=\inf \{\beta > 0: {\cal M}(\beta, 0^{+})>0 \},
	\ee 
and thus, depending on the setting, we can obtain the {\em random quenched} and {\em averaged quenched critical points} denoted by $\beta_c^{\mathrm{rq}}$ and $\beta_c^{\mathrm{aq}}$, respectively. When, for either of the two, $0<\beta_{c} <\infty$, we say that the system undergoes a {\em phase transition} at $\beta=\beta_{c}$.

In the next theorem we collect some results taken from \cite{DM,DGH}, that guarantee the existence of the thermodynamic limit of the quantities previously defined in the random quenched setting. In order to state the existence of the limit magnetization,  we define the set
	\be
	\mathcal{U}^{rq}:=\, \left\{\left(\beta, B\right): \beta\ge 0, B\neq0 \; \mbox{or} \, \; 
	0< \beta < \beta_c^{\mathrm{rq}}, B=0\right\}
	\ee
where $\beta_c^{\mathrm{rq}}$ is the random quenched critical value which is identified in the next theorem:

\begin{theorem}[Thermodynamic limits for the random quenched law \cite{DM,DGH}]\label{term_lim}
Assume that the random graph sequence $\left(G_{\sss N}\right)_{N\geq1}$ is locally tree-like, 
uniformly sparse, and with asymptotic degree distribution $D$ with strongly finite mean, then the following conclusions hold:
\begin{itemize}

\item[(i)] For all $0\le \beta < \infty$ and $B\in \mathbb{R}$, the quenched pressure exists almost surely in the thermodynamic limit $N\to\infty$ and is given by
	\begin{equation}\label{lim_press}
	\psi (\beta, B)  \,:=\,  \lim_{N\rightarrow \infty} \psi_{\sss N} (\beta, B).
	\end{equation}
Moreover, $\psi (\beta, B)$ is a non-random quantity.

\item[(ii)]   For all $\left(\beta, B\right) \in \mathcal{U}^{rq}$,
the \emph{random quenched magnetization per vertex} exists almost surely in the limit $N\to\infty$
and is given by
	\begin{equation}\label{lim_magn}
	M(\beta, B) \,:=\, \lim_{N\rightarrow \infty} M_{\sss N} (\beta, B).
	\end{equation}
The limit value $M(\beta, B)$ equals: $M(\beta, B) \,=\, \frac{\partial}{\partial B} \psi (\beta, B)\,$ for 
$B\neq 0$, whereas $M(\beta, B)=0$ in the region $0< \beta < \beta_c^{\mathrm{rq}}$, $B=0$\;.

\item[(iii)] The critical inverse temperature is given by
	\be
	\label{betac}
	\beta_c^{\mathrm{rq}} = \atanh\left(1/\nu\right),
	\ee
where $\nu$ is defined in (\ref{nu_d}).

\item[(iv)]  For all $\left(\beta, B\right) \in \mathcal{U}^{rq}$, the thermodynamic limit of susceptibility exists almost surely and is given by
\begin{equation}\label{lim_susc}
\chi(\beta,B) \,:=\, \lim_{N\rightarrow \infty} \chi_{\sss N} (\beta, B) \,=\, \frac{\partial^2}{\partial B^2} \psi (\beta, B).
\end{equation}
\end{itemize}
\end{theorem}

\noindent
Let us remark that, since  $\nu\le 1$ for both  $\CMNtwo$ and $\CMNonetwo$, from \eqref{betac}  it follows  that $\beta^{\mathrm{rq}}_c= \infty$, which means that  there is no quenched phase transition in these models. On the other hand, it is interesting to note that $\CMNtwo$ is critical as far as the existence of the giant component is concerned  \cite{vdH}, since it has $\nu=1$.

The next corollary states the existence of the limit of the thermodynamic quantities in the averaged quenched setting, showing also that averaged quenched pressure and magnetization coincide with their random quenched  counterparts. As a consequence, the averaged quenched critical inverse temperature $\beta^{\mathrm{aq}}_{c}$ coincides with the random quenched one, i.e., $\beta^{\mathrm{aq}}_{c}=\beta_{c}^{\mathrm{rq}}$. We denote this unique critical value by $\beta_c^{\mathrm{qu}}$, and denote the uniqueness set by $\mathcal{U}^{\mathrm{qu}}$.

\begin{corollary}[Thermodynamic limits for the averaged quenched law]\label{corol_avquetermlim}
Under the assumptions of Theorem \ref{term_lim}, the following conclusions hold.
\begin{itemize}
\item[(i)] The thermodynamic limit of the averaged quenched pressure exists almost surely and is given by
	\begin{equation}\label{aver_lim_press}
	\overline{\psi} (\beta, B)  \,:=\,  \lim_{N\rightarrow \infty} \overline{\psi}_{\sss N} (\beta, B)  
	\,=\, \psi(\beta, B).
	\end{equation}

\item[(ii)]   For all $\left(\beta, B\right) \in \mathcal{U}^{rq}$, the thermodynamic limit of the averaged quenched magnetization exists almost surely and is given by
	\begin{equation}\label{aver_lim_magn}
	\overline{M}(\beta, B) \,:=\, \lim_{N\rightarrow \infty} \overline{M}_{\sss N} (\beta, B) \,=\,M(\beta, B) .
	\end{equation}

\item[(iii)]  For all $\left(\beta, B\right) \in \mathcal{U}^{\mathrm{qu}}$
	\begin{equation}\label{aver_lim_susc}
	\liminf_{N \to \infty}\overline{\chi}_{\sss N}(\beta,B) \, \geq \,  \chi (\beta, B).
	\end{equation}
\end{itemize}
\end{corollary}

\begin{proof}
Since $\overline{\psi}_{\sss N} (\beta, B) \,=\, \q \left(\psi_{\sss N} (\beta, B)\right)$, by  item $(i)$ of Theorem \ref{term_lim} and the Bounded Convergence Theorem we obtain $\overline{\psi} (\beta, B) \,=\, \psi (\beta, B)$. Item $(ii)$ can be proved in  the same way.
The proof of  the statement $(iii)$ is a straightforward  consequence of the law of total variance:
	\be\label{cond-var}
	\text{Var}_{\p} \left(\frac{S_{\sss N}}{\sqrt{N}}\right) = 
	\q \left( \text{Var}_{\m} \left(\frac{S_{\sss N}}{\sqrt{N}}\right)     \right) +
 	\text{Var}_{\q} \left(\m\left(\frac{S_{\sss N}}{\sqrt{N}}\right) \right)\;,
	\ee
i.e., using the thermodynamic quantities in Definition \ref{def_th_quant},
\be\label{cond-var-th-q}
	\overline{\chi}_{\sss N} = 
	\q \left( {\chi}_{\sss N}\right) +
 	\text{Var}_{\q} \left(\sqrt{N}M_N \right).
	\ee
	In  fact, recalling the almost sure convergence of the non negative variables $ {\chi}_{\sss N}(\beta , B)$ to the {\em non random quantity} $\chi (\beta, B)$, see \eqref{lim_susc}, and applying  Fatou's lemma, we have
	$\liminf_{N \to \infty}{\q}\left ({\chi}_{\sss N}\right) \ge  \chi $ that, with \eqref{cond-var-th-q}, gives the thesis.  
	
\end{proof}

\subsection{LLN and CLT in random quenched setting}
\label{sec-RQ-setting}
We will study the asymptotic behavior of the total spin $S_{\sss N}$ under different
scalings in the random quenched setting.  In \cite{DGH} it is proved that 
this sum normalized by $N$ converges almost surely to a number that is the magnetization of the model (\ref{lim_magn}).
As a preliminary step we will prove here a similar result with a different approach 
based on large deviation theory, which leads to exponentially fast convergence in probability.   Such kind of convergence is defined as follows.

\begin{definition}[Exponential convergence]
\label{def_expconverg}
We say that a sequence of random variables $X_{\sss N}$ with laws $\mu_{\sss N}$ \emph{converges in probability exponentially fast} to a constant $x_0$ w.r.t. 
 $\mu_{\sss N}$   and we write $X_{\sss N} \stackrel{\exp}{\longrightarrow} x_0$, if for any $\varepsilon > 0$ there exists a number $L=L(\varepsilon) > 0$ such that
$$\mu_{\sss N} \left(\left | X_{\sss N} - x_0\right| \geq \varepsilon \right) \leq \e^{-NL}  \quad \mbox{for all sufficiently large N.}$$ 
 \end{definition}

\noindent
Our first result in the random quenched setting is the Strong Law of Large Number for $S_{\sss N}/N$.
\begin{theorem}[Random quenched SLLN]
\label{random_que_SLLN}
\col{Let   $\left(G_{\sss N}\right)_{N\geq1}$ be a sequence of random graphs that are locally tree-like, 
uniformly sparse, and with asymptotic degree distribution with strongly finite mean. 
Then, for all $\left(\beta, B\right) \in \mathcal{U}^{\mathrm{qu}}$}
	\begin{equation}\nonumber
	\frac{S_{\sss N}}{N} \stackrel{\exp}{\longrightarrow} M \qquad 
	\mbox{w.r.t.} \;\; \m, \; \quad \mbox{as}\; \; N \rightarrow \infty,
	\end{equation}
where $M= M(\beta,B)$ is defined in \eqref{lim_magn}.
\end{theorem}

\noindent
To see fluctuations of the total spin, one needs to rescale $S_{\sss N}-NM_{\sss N}$ by $\sqrt{N}$. In this case, restricting to the set  $\mathcal{U}^{\mathrm{qu}}$ of parameters  $(\beta,B)$ such that there exits a unique Gibbs measure \cite{MMS}, one obtains a Gaussian random variable in the limit $N\to\infty$. We deduce the random quenched CLT in the next theorem:

\begin{theorem}[Random quenched CLT]
\label{random-CLT}
\col{Let   $\left(G_{\sss N}\right)_{N\geq1}$ be a sequence of random graphs that are locally tree-like, 
uniformly sparse, and with asymptotic degree distribution with strongly finite mean. 
Then, for  all $\left(\beta, B\right) \in \mathcal{U}^{\mathrm{qu}}$} 
	\begin{equation}\nonumber
	\frac{S_{\sss N} - N M_{\sss N} }{\sqrt{N}} \; 
	\stackrel{{\cal D}}{\longrightarrow} \; \mathcal{N}\left(0,\chi\right), 
	\qquad \mbox{w.r.t.} \;\; \m, \; \quad \quad \mbox{as}\; \; N \rightarrow \infty,
	\end{equation}
where $\chi=\chi(\beta,B)$ is defined in \eqref{lim_susc} and $\mathcal{N}\left(0,\chi\right)$ denotes a centered Gaussian random variable with variance $\chi$.
\end{theorem}

\subsection{LLN and CLT in averaged quenched setting}
\label{sec-AQ-setting}
In the averaged quenched setting, after presenting the Weak Law of Large Numbers, we formulate a Central Limit Theorem  for  $\CMNtwo$ and $\CMNonetwo$ models.\\ From the $\CMNonetwo$ example we see that while averaging with respect to the graph measure $\q$ does not change $\beta_c^{\mathrm{qu}}$ nor $\psi (\beta, B)$, (cf. Corollary \ref{corol_avquetermlim}),  it does change the variance of the limiting Gaussian distribution of the total spin since, when passing from random to averaged quenched measure, the fluctuations of the graph are taken into account. 
Thus, Theorem \ref{CLT_an2} shows that the total spin rescaled by the square root of the volume also has fluctuations with respect to the random graph measure. Those fluctuations, quantified by $\sigma_{\sss G}^{2}$ in Theorem \ref{CLT_an2}, are in turn forced by the unequal degrees in $\CMNonetwo$.

\begin{theorem}[Averaged quenched WLLN]
\label{wlln_a}
\col{Assume that the law $Q_N$ of the random graph is such that 
almost all sequences $\left(G_{\sss N}\right)_{N\geq1}$ are locally tree-like,
uniformly sparse, and with asymptotic degree distribution with strongly finite mean.
%
Then, for all $\left(\beta, B\right) \in \mathcal{U}^{\mathrm{qu}}$}
	\begin{equation}\nonumber
	\frac{S_{\sss N}}{N} \stackrel{\mathbb{P}}{\longrightarrow} M(\beta, B)  
	\qquad \mbox{w.r.t.} \;\; \p, \; 	\quad \mbox{as} \; \; N \rightarrow \infty
	\end{equation}
where $ \stackrel{\mathbb{P}}{\longrightarrow} $ denotes convergence in probability,
i.e., for all $\varepsilon >0$,
	\be
	\lim_{N \rightarrow \infty} \p\left(\left|\frac{S_{\sss N}}{N} - M(\beta, B) \right| > \varepsilon \right) 
	= 0.
	\ee
\end{theorem}

\noindent
The proof of this theorem follows from the random quenched SLLN. Indeed,  by the definition of the average quenched measure, 
	\be
	\p\left(\left|\frac{S_{\sss N}}{N} - M(\beta, B) \right| > \varepsilon \right)\, 
	= \, \q\left[\m \left(\left|\frac{S_{\sss N}}{N} - M(\beta, B) \right| > \varepsilon \right)\right],
	\ee
that, combined with Theorem \ref{random_que_SLLN}, i.e.  $\lim_{N \rightarrow \infty} \m \left(\left|\frac{S_{\sss N}}{N} - M(\beta, B) \right| > \varepsilon \right)= 0$ and the Bounded Convergence Theorem leads to the result. We can only prove a {\em weak} LLN, since exponential convergence as in Theorem \ref{random_que_SLLN} does not hold. See Section \ref{sec-no-expon-conv-aq} for more details.

\begin{theorem}[Averaged quenched CLT for $\CMNtwo$]
\label{CLT_an1}
\col{Let  $\gs$ be sequences  of $\CMNtwo$ graphs}. Then, \col{for any $\beta \ge 0$} and
$B\in\mathbb{R}$
	\begin{equation}\nonumber
	\frac{S_{\sss N} - \p \left(S_{\sss N}\right) }{\sqrt{N}} \; 
	\stackrel{{\cal D}}{\longrightarrow} \; \mathcal{N}\left(0, \chi\right), 
	\qquad \mbox{w.r.t.} \;\; \p, \; \quad \mbox{as} \; \; N \rightarrow \infty,
\end{equation}
where $\chi=\chi(\beta, B)$ is the thermodynamic limit of the susceptibility \eqref{lim_susc} specialized to the Ising model on $\CMNtwo$. Moreover, $\chi(\beta, B)$ is also equal to the susceptibility of the one-dimensional Ising model (see Sec \ref{Sec_Ising}), i.e.,
	\be\label{chi_ising}
	\chi(\beta, B)=\chi^{d=1}(\beta,B) = \frac{\cosh(B) \e^{-4\beta}}{(\sinh(B)+\e^{-4\beta})^{3/2}}\; .
	\ee
\end{theorem} 

\noindent
The identification of the variance $\chi(\beta, B)$ as $\chi^{d=1}(\beta,B)$ also holds for the random quenched setting in Theorem \ref{random-CLT}, so that we see that, for $\CMNtwo$, the averaged quenched and random quenched variances in the CLT are equal. We next investigate a case where this is not true:

\vspace{0.4cm}
\begin{theorem}[Averaged quenched CLT for $\CMNonetwo$]
\label{CLT_an2}
\col{Let  $\gs$ be sequences  of  $\CMNonetwo$  graphs}.  Then,  \col{for any $\beta \ge 0$} and
$B\in\mathbb{R}$
	\begin{equation}\nonumber
	\frac{S_{\sss N} - \p \left(S_{\sss N}\right) }{\sqrt{N}} \; 
	\stackrel{{\cal D}}{\longrightarrow} \; \mathcal{N}\left(0,\sigma_{\mathrm{aq}}^2\right), 
	\qquad \mbox{w.r.t.} \;\; \p, \; \quad \mbox{as} \; \; N \rightarrow \infty,
	\end{equation}
where $\sigma_{\mathrm{aq}}^2=\chi +\sigma_{\sss G}^2$, with $\chi=\chi(\beta, B)$ the thermodynamic limit of the susceptibility of the Ising model on $\CMNonetwo$ (whose explicit expression is given in \eqref{var12-random-quenched} below) and $\sigma_{\sss G}^2=\sigma_{\sss G}^2(\beta,B)$ a positive number that is defined in \eqref{sigmag2} below.
\end{theorem}


\subsection{Discussion}
\label{discussio}

\subsubsection{On the rate of convergence for the laws of large numbers}
\label{sec-no-expon-conv-aq}
In the random quenched setting, we prove that the finite volume total spin $S_{\sss N}$ normalized by $N$ converges in probability exponentially fast to the magnetization $M$, which is a non-random quantity.
For the averaged quenched setting, instead, it is easy to give examples for which the convergence at exponential rate is lost. Let us consider the sequence $\gs$ given by
	\begin{equation} 
	G_{\sss N}\,=\, \left\{\begin{array}{ll} \CMNreg \quad \quad \quad \; 
	\mbox{with probability} \; 1 - \frac{1}{N},\\ 
	K_{\sss N} \quad \quad \quad \quad \quad \; \; \mbox{with probability} \; \frac{1}{N}, 
	\end{array}\right.  
	\end{equation}
where $\CMNreg$ is the $r$-regular random graph ($r\in\mathbb{N}$), i.e., a configuration model with degree  sequence $d_{i}=r$ for all $i\in[N]$, and $K_{\sss N}$ is the complete graph. The infinite-volume magnetization $M$ of the Ising model \eqref{bg} on the sequence $\gs$  coincides with the limiting  magnetization of $\CMNreg$, hence $M\ne 1$ \col{for any $\beta < \infty$}.  On the other hand, the infinite-volume magnetization  on the sequence of complete graphs $(K_{\sss N})_{N\ge 1}$ is 1 and, thus, there exists $\varepsilon > 0$ such that 
	\[\p\left(\left| \left.\frac{S_{\sss N}}{N} - M\right| > \varepsilon \; \right| \; 
	G_{\sss N} = K_{\sss N} \right)=1 + o(1),
	\]
as $N\to \infty$. Therefore,
	\begin{align}
	\p\left(\left|\frac{S_{\sss N}}{N} - M\right| > \varepsilon \right) 
	&= \p\left(G_{\sss N} = K_{\sss N}\right)\cdot 
	\p\left(\left| \left.\frac{S_{\sss N}}{N} - M\right| > \varepsilon \; 
	\right| \; G_{\sss N} = K_{\sss N} \right)\\ \nonumber
	&\quad+ \p\left(G_{\sss N} = \col{\CMNreg} \right)\cdot
	\p\left(\left| \left.\frac{S_{\sss N}}{N} - M\right| > \varepsilon \; 
	\right| \; G_{\sss N} = \col{\CMNreg} \right) \\ \nonumber
	& \geq \frac{1}{N}\left(1 + o(1)\right),
	\end{align}
which prevents the sequence from converging exponentially fast. 

%

\subsubsection{CLT proof strategy}\label{sec_proof_strategy}
The primary purpose of this paper is to prove Central Limit Theorems in the two different settings introduced in Section \ref{iniz_def}.  The proofs are  different  (in fact they require the control of the  fluctuations of $S_{\sss N}$ with respect to different ensembles), but are based on the same main idea \cite{El}. This idea consists in using the moment generating function of the random variables $\frac{S_{\sss N} - \mathbb{E} \left(S_{\sss N}\right)}{\sqrt{N}}$, where $\mathbb{E}$ is the  average in the chosen measure, and in showing that with respect to the same measure these moment generating functions converge to  those of Gaussian random variables. Whenever possible, this step requires the computation of the variances of the \col {limiting} Gaussian variables.  This is achieved by considering the 
\col{scaled}
\emph{cumulant generating functions} of $S_{\sss N}$, given by
	\begin{equation}
	\label{c_N}
	c_{\sss N}(t) = \frac{1}{N} \log{\m \left[\exp\left(t S_{\sss N}\right)\right]}\;
	\end{equation}
in the random quenched setting and
	\begin{equation}
	\label{tildec_N}
	\overline{c}_{\sss N}(t)= \frac{1}{N} \log{\p \left[\exp\left(t S_{\sss N}\right)\right]}
	\end{equation}
in the averaged quenched setting. By taking the  second derivative of  \eqref{c_N} and \eqref{tildec_N}
and evaluating it in zero one obtains, respectively, $\text{Var}_{\m} \left(\frac{S_{\sss N}}{\sqrt{N}}\right)$ and $\text{Var}_{\p} \left(\frac{S_{\sss N}}{\sqrt{N}}\right)$. The crucial argument in the proof of the theorems is to show the existence of these variances in the limit $N\to\infty$. This in turn is a consequence of the existence of the limit of the sequences $(c_{\sss N}(t))_{N\ge1}$ and $(\overline{c}_{\sss N}(t))_{N\ge 1}$ and the sequences of their second derivatives for $t=t_{\sss N}=o(1)$.

While the existence of the limit $c(t) := \lim_{N\to\infty} c_{\sss N}(t)$ can be established for the Ising model on locally tree-like random graphs as a simple consequence of  the existence of the random quenched pressure, the existence of the limit $\overline{c}(t) := \lim_{N\to\infty} \overline{c}_{\sss N}(t)$ is more challenging. In particular, it requires detailed knowledge not only of the typical local structure of the graph around a random vertex, but also of the fluctuations around that structure. Moreover, the general argument of \cite{El} relies on concavity of the first derivatives of the cumulant generating functions. This can be achieved for $c_{\sss N}(t)$, i.e., in the random quenched setting, thanks to the GHS inequality, which holds for the ferromagnetic Boltzmann-Gibbs  measures $\mu_{G_{\sss N}}$. On the other hand, the GHS inequality is in general not known for the averaged quenched measure. 

\col{Thus in the averaged quenched setting} we focus on two specific models, 
i.e. the  $\CMNtwo$ and $\CMNonetwo$ random graphs, which allow for explicit computations of the relevant quantities, even in the averaged quenched setting.  In fact in these cases, since the typical structure in the graphs  are cycles (for $\CMNtwo$) and line and cycles (for $\CMNonetwo$),  the \col{
averaged quenched pressure of the Ising model on these graphs can be expressed in terms of the
the Ising model pressure $\psi^{d=1}(\beta, B)$ of the  one-dimensional nearest-neighbour Ising model}. 
\col{It turn out however that whereas
the averaged quenched  pressure  of  the regular random graph  $\CMNtwo$  
exactly equals $\psi^{d=1}(\beta, B)$,  in the case of $\CMNonetwo$ the pressure  is more involved, see \eqref{pressQcaso2}.} Indeed, besides $\psi^{d=1}(\beta, B)$, a new term appears that depends on a set of random variables $(p^{\sss (N)}_{\ell})_{\ell\geq 1}$ whose value 
depends on the realization of the random graph. More precisely, $N p^{\sss (N)}_{\ell}$ is \col{the} number of  lines of length $\ell$ in the graph (the cycles give a vanishing contribution in the thermodynamic limit).  Then, in order to prove the CLT for $\CMNonetwo$, it is of pivotal importance to control the fluctuations of the random variables  $p^{\sss (N)}_{\ell}$ in the thermodynamic limit. This result is obtained in \cite{dPB}, where it is proven that the joint limit law of the number of connected components in a graph with vertices of degrees 1 and 2 is Gaussian. Relying on this result, we can complete our proof of the averaged quenched CLT for $\CMNonetwo$.

\subsubsection{Differences between random quenched and averaged quenched setting}
\label{diff_rq_aq}
We explain here the distinction between the random quenched and the averaged quenched CLTs. The variance $\text{Var}_{\p} \left(\frac{S_{\sss N}}{\sqrt{N}}\right)$, expressing the fluctuations of $S_{\sss N}/{\sqrt{N}}$ with respect to the averaged quenched measure, has two contributions. The first is given by the average over random graphs of the conditional 
variance of $S_{\sss N}/{\sqrt{N}}$ with respect to the Boltzmann-Gibbs measure 
(the conditioning is given by the graph realization);
the second contribution is given by the variance of conditional mean  of $S_{\sss N}/{\sqrt{N}}$ (see (\ref{cond-var})).

If a CLT with respect to the random quenched measure holds, then the thermodynamic limit
of the first term in the right-hand side of \eqref{cond-var} equals the magnetic susceptibility $\chi$,
which is a self-averaging quantity. It is clear from \eqref{cond-var} that one expects a different variance in the CLT in the averaged quenched case whenever the thermodynamic limit of the second term on  the right-hand side is different from zero.

The analysis that we perform in Section \ref{three} on the configuration model leads us to conjecture
that when the vertex degrees are not all equal, one has that $\sigma^2 := \lim_{N\to\infty} \text{Var}_{\q} \left(\m\left(\frac{S_{\sss N}}{\sqrt{N}}\right)\right)$ is a strictly positive number. On the contrary, we predict that when the degrees are fixed, such as in the $r$-regular random graph, the two limiting variances are equals, thus yielding no distinction between random and averaged quenched central limit theorems. We prove this in the case where $r=2$.

Furthermore, when there are non-vanishing fluctuations of the rescaled total spin with respect to the graph measure, those will be determined by the fluctuations of the degrees distribution of the graph. Those, in turn, are determined by the explicit graph construction. For instance, in the case of the $\CMNonetwo$ random graph with a fixed number of vertices with degree 1 the limiting variance is obtained in \eqref{vargk}. With reference to the alternative construction of the $\CMNonetwo$ graph described in Remark \ref{alt_costr}, where the number of vertices of degree 1 (called $n_1$) follows a Binomial distribution, a different limiting variance would arise, despite the fact that asymptotically the two models have the same degrees distribution. For instance,  the deterministic factor $\frac{1-p}{2}$ in formula \eqref{vargk}, which arises from $n_1/2$, \col{in the alternative construction} would be replaced by a random quantity with non-vanishing fluctuations.

\subsubsection{Violation of CLT at $\beta_c$}
\label{sec-non-classical-limits}
In the case of the random quenched CLT, the variance is given by the susceptibility of the model. If $B=0$ and $\beta= \beta_c^{\mathrm{qu}}$ the variance diverges because the susceptibility becomes infinite at the critical temperature. So, for these parameters, the CLT breaks down and a different scaling of the total spin $S_{\sss N}$ is needed to obtain a non-trivial limiting distribution. For example, in \cite{El}, it is shown that for the Curie-Weiss model, at $B=0, \beta= \beta_c$, the quantity $\displaystyle \frac{S_{\sss N}}{N^{3/4}}$ converges in distribution to the random variable $X$ with density proportional to $\exp{\displaystyle \left(-x^4/12\right)}$. We do not investigate this problem in this paper. We conjecture a similar behavior for the Ising model on locally tree-like random graphs with finite fourth moment of the degree distribution. Those graphs have been shown to be in the same universality class as the Curie-Weiss model, i.e., their critical exponents agree with the mean-field critical exponents of the Curie-Weiss model \cite{DGH2}. We investigate this problem in detail for another class of locally tree-like random graphs, namely, generalized random graphs (recall \cite[Chapter 6]{vdH}) in a forthcoming paper \cite{DGGHP15}. Interestingly, as in \cite{DGH2}, the mean-field behavior breaks down when the fourth moment of the degrees in the random graph becomes infinite. In the latter case, another limit theorem holds for the total spin.

\subsubsection{Low-temperature region}
\label{sec-low-T}
When the graph sequence $\left(G_{\sss N}\right)_{N\ge 1}$ is such that there exists a finite critical inverse temperature $\beta_c$ (cf.\ \eqref{betac}), then Theorems \ref{random_que_SLLN}, \ref{random-CLT}, \ref{wlln_a} do not apply to the low temperature region corresponding to $B=0$, $\beta>\beta_c$. 
For the Ising model on $\mathbb{Z}^d$ \cite{Chuck, LOF} or on the complete graphs \cite{El}, in the low temperature region the law of large number for the empirical sum of the spin breaks down because, in the thermodynamic limit,  the Boltzmann-Gibbs measure becomes a mixture of two pure states $\mu = \frac12 (\mu^+ + \mu^-)$. As a consequence, the empirical magnetization is distributed like the sum of two Dirac deltas at the symmetric values $\pm M^*(\beta)$, with  $M^*(\beta)= \lim_{B\searrow 0} M(\beta,B)$, whereas the CLT with respect to
the Boltzmann-Gibbs measure breaks down. However, by using general properties of ferromagnetic systems (e.g. GKS inequalities)
it is possible to prove a CLT with respect to the measure $\mu^+$, respectively $\mu^-$ \cite{Chuck, LOF}. 
For the Ising model on random graphs it is believed that a similar picture apply.
For instance, for the Ising model on regular random graphs it has been proved in \cite{MMS}
that the low temperature measure is a convex combination of $+$ and $-$ states
and therefore, by appealing to the general results of \cite{Chuck} we conclude
that a CLT \col{holds} in the pure phase. 

\section{Proofs of random quenched results}
\label{two}
In order to prove the Strong Law of Large Numbers in the random quenched setting, we present a preliminary theorem that guarantees exponential convergence, see Definition \ref{def_expconverg},  under general hypotheses.

\subsection{Exponential convergence}
Let ${\mathcal W}= \left(W_n\right)_{n \geq 1}$ be a sequence of random vectors which are defined on probability spaces $\left\{\left(\Omega_n, {\mathcal F}_n, P_n \right)\right\}_{n \geq 1}$ and which take values in $\mathbb R^D$. We define the cumulant generating functions as
	\begin{equation}\label{funz_c}
	c_n(t) = \frac{1}{a_n} \log E_n \left[ \exp(\left\langle  t, W_n \right\rangle)\right], 
	\quad \quad n=1,2,\ldots, \quad t \in \mathbb R^D,
	\end{equation}
where $(a_n)_{n \geq 1}$ is a sequence of positive real numbers tending to infinity, $E_n$ denotes expectation with respect to $P_n$, and $\left\langle -,-\right\rangle$ is the Euclidean inner product on $\mathbb R^D$. We assume that the following hypotheses hold:
	\begin{enumerate}
	\item[(a)] Each function $c_n(t)$ is finite for all $t \in \mathbb R^D$;
	\item[(b)] $c(t) = \lim_{n \rightarrow \infty} c_n(t)$ exists for all $t \in \mathbb R^D$ and is finite.
	\end{enumerate}

\begin{theorem}[Exponential convergence and cumulant generating functions]
\label{II.6.3}
Assume hypotheses $(a)$ and $(b)$. Then the following statements are equivalent:
	\begin{itemize}
	\item[(1)] $W_n/a_n \stackrel{\exp}{\longrightarrow} z_0$;
	\item[(2)] $c(t)$ is differentiable at $t=0$ and $\nabla c(0) = z_0$.
	\end{itemize}
\end{theorem}
\noindent
See \cite[Theorem II.6.3]{El} for a proof based on a large deviation argument.

\subsection{Random quenched SLLN: Proof of Theorem \ref{random_que_SLLN}} \label{proof_slln_q}
According to Theorem \ref{II.6.3}, exponential convergence can be obtained by proving the existence of the limit of the random quenched cumulant generating function $c_{\sss N}(t)$ defined in (\ref{c_N}), and proving the differentiability of the limiting function  in $t=0$. We have
	\be
	\label{c_randomq_zeta}
	c_{\sss N}(t)  \, = \, \frac{1}{N} \log \sum_{\sigma \in \Omega_{\sss N}} 
	\frac{\exp \left(\beta \sum_{(i,j)\in e_{\sss N}}{\sigma_i \sigma_j} 
	+ \left(B + t\right)\sum_{i \in[N]} {\sigma_i}  \right)}{Z_{\sss N}(\beta,B)}
	= \, \frac{1}{N} \log \frac{Z_{\sss N}(\beta,B + t)}{Z_{\sss N}(\beta,B)},  
	\ee
and recalling the definition  (\ref{press_N}), the function $c_{\sss N}(t)$ can be rewritten as a difference of random quenched pressures as
	\begin{equation}\label{c_finito}
	c_{\sss N}(t) = \psi_{\sss N}(\beta , B + t) - \psi_{\sss N} (\beta, B).
	\end{equation}
The existence of the limit 
	\be\label{c_infinito}
	c(t):= \lim_{N \rightarrow \infty} c_{\sss N}(t) = \psi(\beta , B + t) - \psi (\beta, B) \quad a.s.
	\ee
is then obtained from the existence of the pressure in the thermodynamic limit, as stated in Theorem \ref{term_lim}. Moreover,  from the differentiability of the infinite-volume random pressure with respect to $B$,  we also obtain 
	\[
	c'(t) = \frac{\partial}{\partial t} \left[\psi(\beta , B + t) - \psi (\beta, B)\right] 
	= \frac{\partial}{\partial B} \left[\psi(\beta , B + t)\right],
	\]
and hence	
	\[
	c'(0) = \frac{\partial}{\partial B} \left[\psi(\beta , B)\right] = M(\beta, B).
	\]
Thus, by Theorem \ref{II.6.3}, we obtain the exponential convergence in Theorem \ref{random_que_SLLN}, which immediately implies the SLLN. \qed

\subsection{Random quenched CLT: Proof of Theorem \ref{random-CLT}}
\label{prf_rq_CLT}
We give the proof for $B\geq0$ only, the case $B<0$ is handled similarly. The strategy of the proof is to show that w.r.t.\ the random quenched measure the moment generating 
function of the random variable 
	\begin{equation}\label{V}
	V_{\sss N} = \frac{S_{\sss N} - N M_{\sss N}(\beta,B)}{\sqrt{N}}.
	\end{equation}
converges to the moment generating function of a Gaussian random variable with
variance $\chi(\beta,B)$ given in \eqref{lim_susc}, i.e.,
	\begin{equation}\label{funzGener}
	\lim_{N \rightarrow\infty} 
	\m \left(\exp\left(t V_{\sss N}\right)\right) 
	\;=\; \exp\left(\frac{1}{2}\chi(\beta,B) t^2\right) \qquad \mbox{for all}\; t \in \left[0,\alpha\right),
	\end{equation}
and  some $\alpha >0$.  This can be done by  expressing  $\m \left(\exp\left(t V_{\sss N}\right)\right)$ in terms of the second derivative of the cumulant generating function  $c_{\sss N}(t)$ defined in \eqref{c_N}.\\
\noindent
A simple computation shows that 
\begin{equation}\label{c'_N}
c'_{\sss N}(t) \,= \,\frac{1}{N} \frac{\m \left(S_{\sss N}\exp\left(t S_{\sss N}\right)\right)}{\m \left(\exp\left(t S_{\sss N}\right)\right)} \,=\, {\m}_{(\beta, B+t)}\left(\frac{S_{\sss N}}{N}\right),
\end{equation}
where, in order to stress the dependence on the magnetic field, we have used the symbol ${\m}_{(\beta, B+t)} \left(\cdot\right)$ to denote the $\m$-average in the presence of the field $B+t$. Thus,
	\eqan{
	\label{c''_N}
	c''_{\sss N}(t)&= \, \frac{1}{N}  \frac{\m \left(S_{\sss N}^2\exp\left(t S_{\sss N}\right)\right)
	\m \left(\exp\left(t S_{\sss N}\right)\right) - \m ^2\left(S_{\sss N}\exp\left(t S_{\sss N}\right)\right)}
	{\m ^2\left(\exp\left(t S_{\sss N}\right)\right)} \\
	&= \, Var_{\m(\beta, B+t)}\left(\frac{S_{\sss N}}{N}\right).\nn
	}
In particular, the derivatives in $t=0$ of $c_{\sss N}(t)$ equal
	\[
	c'_{\sss N}(0) \,=\, M_{\sss N}(\beta, B),
	\]
and
	\[
	c''_{\sss N}(0)\,=\, \frac{1}{N} \left[\m \left(S_{\sss N}^2\right) -  \m ^2\left(S_{\sss N}\right)\right] 
	\,=\, \chi_{\sss N}(\beta,B). 
	\]
By Theorem \ref{term_lim},
 	\begin{equation}\label{carab1}
	c''_{\sss N}(0) \,=\, \chi_{\sss N}(\beta,B) \,\rightarrow \, \chi(\beta,B) \qquad \qquad \mbox{as}\;	
	\quad N \rightarrow\infty.
	\end{equation}
Let us take $t>0$ and set $t_{\sss N} = t/\sqrt{N}$. By using the fact that $c_{\sss N}(0)=0$ and applying Taylor's theorem with Lagrange remainder, we obtain
	\begin{align}\label{taylor}
	\log \m \left(\exp\left(t V_{\sss N}\right)\right) 
	&=\, \log \m \left(\exp\left(\frac{t S_{\sss N} - tN M_{\sss N}(\beta,B)}{\sqrt{N}}\right)\right)\nn\\ 
	&= \log \left[\m \left(\exp\left(t_{\sss N} S_{\sss N}\right)\right)\right] 
	- t \sqrt{N} M_{\sss N}(\beta,B)\nn\\
	&= N \left[c_{\sss N}(t_{\sss N}) - t_{\sss N} c'_{\sss N} (0)\right]\, 
	= \, N \frac{t_{\sss N}^2}{2} c_{\sss N}''(t_{\sss N}^*)= \frac{t^2}{2} c_{\sss N}''(t_{\sss N}^*),
	\end{align}
for some $t_{\sss N}^* \in [0, t/\sqrt{N}]$.

In order to control the limiting behavior of $c_{\sss N}''(t_{\sss N}^*)$, we exploit the following property of $c_{\sss N}(t)$:
\begin{proposition}[Convergence of double derivative cumulant generating function]
\label{propp}
For \col{$0 < \beta < \beta_c^{\mathrm{qu}}$} and $B\geq 0$, \col{there exists some $\alpha>0$ such that 
$\lim_{N\rightarrow\infty} c''_{\sss N}(t_{\sss N}) = \chi(\beta, B)$ for all $t_{\sss N}\in\left[0, \alpha\right)$
with $t_{\sss N}\to 0$  and almost all sequences of graphs $\gs$.}
\end{proposition}

\noindent {\it Proof of Theorem \ref{random-CLT}.}
Proposition \ref{propp} immediately implies that \eqref{funzGener} holds,
which in turn proves Theorem \ref{random-CLT}.
\qed
\vskip0.5cm

\noindent
The remainder of this section is devoted to the proof of Proposition \ref{propp}. It relies on the concavity of the functions $c^{\prime}_{\sss N}(t)$, as proven in the following lemma:

\begin{lemma}\label{concavity}
For $B\geq0$, $c'_{\sss N}(t)$ is concave on $[-B, \infty)$.
\end{lemma}

\begin{proof} For $B\geq 0$, the concavity of $c'_{\sss N}(t)$ on $[0, \infty)$ can be obtained by observing that this function is the magnetization per particle w.r.t.\ $\m$  in the presence of the magnetic field $B+t$, see  (\ref{c'_N}), and then by applying the GHS inequality. Indeed, for $t \geq -B$,  the GHS inequality \cite[Lemma 2.2]{DGH} implies that
	\begin{equation}
	\frac{\partial^2}{\partial t^2} c'_{\sss N}(t) 
	\, = \, \frac{1}{N} \sum_{i \in [N]} \frac{\partial^2}{\partial B^2}  
	{\m}_{(\beta, B+t)}\left(\sigma_i\right) \leq 0,
	\end{equation}
so that $c'_{\sss N}(t)$\, is concave on\, $\left[- B, \infty \right)$ \, for \, $B\geq0$.
\end{proof}

The proof of Proposition \ref{propp} further requires the following lemma that is proved in \cite[
Lemma V.7.5]{El}:

%
\begin{lemma}[Convergence of derivatives of convex functions]
\label{lemmaV.7.5}
Let $(f_n)_{n\geq1}$ be a sequence of convex functions on an open interval $A$ of $\mathbb R$ such that $f(t) = \lim_{n \rightarrow \infty} f_n(t)$ exists for every $t \in A$. Let $(t_n)_{n\geq 1}$  be a sequence in $A$ that converges to a point $t_0 \in A$. If $f'_n(t_n)$ and $f'(t_0)$ exist, then $\lim_{n \rightarrow \infty} f'_n(t_n)$ exists and equals $f'(t_0)$.
\end{lemma}

\begin{proof}[Proof of Proposition \ref{propp}]
In the following proof, we use subscripts to denote the  dependence on $\beta$ and $B$ of the functions $c_{\sss N}(t)$ and $c(t)$. We first consider $B>0$.  
By H\"older's inequality, $t \mapsto c_{N,\beta,B}(t)$ is a convex function on $\mathbb{R}$ and, as noted in (\ref{c_finito}) and (\ref{c_infinito}), $c_{N,\beta, B}(t) \; {\longrightarrow} \; c_{{\beta,B}}(t) = \psi(\beta , B + t) - \psi (\beta, B)$ as $N \rightarrow \infty$. For $B>0$, the derivative  $\frac{\partial}{\partial B}\psi (\beta, B_0)$ exists for all $B_0$ in some neighborhood of $B$. Hence, there exists $\alpha >0$ such that $c'(t) = \frac{\partial}{\partial B}\psi(\beta , B + t)$ exists for all $-\alpha < t < \alpha$. Then, Lemma \ref{lemmaV.7.5} implies that
	\begin{equation}
	c'_{N,\beta,B}(t) \longrightarrow c'_{\beta,B}(t) \; \; \;\; \quad 
	\mbox{for all}\; -\alpha < t < \alpha \, , \; \quad \mbox{as}\; N \rightarrow \infty.
	\end{equation}
\col{According to Lemma \ref{concavity}, each function $-c'_{N,\beta,B}(t)$ is convex on the interval  $[-B, \infty)$, which contains the origin in its interior since $B>0$.}
By (\ref{lim_susc}) we know that $\frac{\partial^2}{\partial B^2} \psi (\beta, B) = \chi(\beta, B)$, hence
	\begin{equation}
	c''_{\beta,B}(0) = \frac{\partial^2}{\partial B^2} \psi (\beta, B) = \chi(\beta, B).
	\end{equation}
This completes the proof of the proposition for $B>0$, because Lemma \ref{lemmaV.7.5} implies that, for any $0\leq t_{\sss N} < \alpha$ with $t_{\sss N} \rightarrow 0$,
	\begin{equation}
	-c''_{N,\beta,B}\left(t_{\sss N}\right) \longrightarrow -c''_{\beta,B}(0)= - \chi(\beta, B) 
	\; \; \;\;\quad 	\mbox{as}\; N \rightarrow \infty. 
	\end{equation}

The proof of Proposition \ref{propp} for $B>0$ used the fact that all the function ${-c'_{N,\beta,B}}(t)$ is convex on an interval that contains the origin in its interior. But for $B=0$, $-c'_{N,\beta,B}(t)$ is convex on $\left[0, \infty \right)$ and by symmetry is concave on $\left(-\infty, 0\right]$. Hence the above proof must be modified. We fix $0 < \beta < \beta_c^{\mathrm{qu}}$ and 
\col{notice} that $\chi(\beta, 0)$ is finite.

Since for $0 < \beta < \beta_c^{\mathrm{qu}}$, $\frac{\partial }{\partial B}\psi(\beta , B)$ exists for all $B$ real, $c_{\beta,0}(t)= \psi(\beta , t) - \psi (\beta, 0)$ is differentiable for all $t$ real. We define new functions
	\begin{equation} 
	h_{\sss N}(t) \,=\, 
	\left\{\begin{array}{ll} -c'_{N,\beta,0}(t) \quad \quad \quad \; \mbox{for} \; t \geq 0,\\ 
	-c''_{N,\beta,0}(0)\cdot t \quad \quad \mbox{for} \; t < 0, 
	\end{array}\right. \quad \quad 
	h(t) \,=\, 
	\left\{\begin{array}{ll} -c'_{\beta,0}(t) \quad \quad \quad \; \mbox{for} \; t \geq 0,\\ 
	- \chi(\beta, 0)\cdot t \quad \quad \mbox{for} \; t < 0. \end{array}\right.  
	\end{equation}
Then $h_{\sss N}$ is continuous at $0$ since $c'_{N,\beta,0}(0) = M(\beta,0) = 0,$ and convex on $\mathbb R$. By (\ref{carab1}), $c''_{N,\beta,B}(0) \rightarrow \chi(\beta, B)$, and so $h_{\sss N}(t) \rightarrow h(t)$ for all $t$ real. \col{Since $\chi(\beta, 0)$ is 
finite}, $h'(0) = - c''_{\beta,0}(0) = - \chi(\beta, 0)$. Lemma \ref{lemmaV.7.5} implies that for any $t_{\sss N} \geq 0$ with $t_{\sss N} \rightarrow 0$
	\begin{equation}
	h'_{\sss N}\left(t_{\sss N}\right) 
	= -c''_{N,\beta,0}\left(t_{\sss N}\right) \, \rightarrow \, h'(0) 
	= - \chi(\beta,0) \quad \quad 
	\mbox{as}\; N \rightarrow \infty.
	\end{equation}
This proves the proposition for $B=0$.
\end{proof}

\section{Proofs for $\CMNtwo$}
\label{three}
In this section, we prove the CLT with respect to the averaged quenched measure for the 2-regular random graph $\CMNtwo$. We start by computing partition function for the one-dimensional Ising model.  We will use them in the proofs of CLT's because the structures formed in $\CMNtwo$ and $\CMNonetwo$, are lines (indicated with the letter $l$) and cycles (or tori, indicated with the letter $t$).

\subsection{Partition functions for the one-dimensional Ising model}
\label{Sec_Ising}
The partition function of the one-dimensional Ising model is given by
	\be
	\label{uno}
	Z^{\sss (b)}_{\sss N}(\beta,B) 
	= \sum_{\sigma_1=\pm 1}\ldots\sum_{\sigma_{\sss N}=\pm 1}
	\exp\left(\beta \sum_{i=1}^N \sigma_i\sigma_{i+1}(b)+  B \sum_{i=1}^N \sigma_i\right)\;,
	\ee
where for {\em periodic boundary conditions} ($b=t$) we put $\sigma_{N+1}(t) = \sigma_1$,
whereas  we put $\sigma_{N+1}(l) = 0$ for {\em free boundary condition}  ($b=l$). We often omit $(\beta, B)$ from the notation and simply write $Z^{\sss (t)}_{\sss N}$ and $Z^{\sss (l)}_{\sss N}$.
Let $D$ be the  $2\times 2$ matrix defined, for $(\sigma_i,\sigma_{i+1})\in \{-1,+1\}^2$, by
	\[
	D_{\sigma_i,\sigma_{i+1}} 
	= \exp \left[\beta \sigma_i \sigma_{i+1}+ \frac{B}{2}(\sigma_i +  \sigma_{i+1}) \right].
	\]
With this definition we may rewrite (\ref{uno})  for periodic boundary conditions in the form
	\begin{equation}
	\label{tre}
	Z^{\sss (t)}_{\sss N} =  \sum_{\sigma_1=\pm 1}\ldots\sum_{\sigma_{\sss N}=\pm 1}
	D_{\sigma_1,\sigma_2}D_{\sigma_2,\sigma_3}\ldots D_{\sigma_{\sss N},\sigma_{N+1}} 
	\, = \, \mbox{Trace}(D^N) \, = \, \lambda_+^N + \lambda_-^N,
	\end{equation}
where $\lambda_+$ and $ \lambda_-$ are the two eigenvalues of $D$, given by
	\be
	\label{lambda_piu_meno}
	\lambda_{\pm} = \lambda_{\pm}(\beta, B) 
	= \e^{\beta } \left[ \cosh(B)\pm \sqrt{\sinh^2(B)+\e^{-4\beta }}\right]\;.
	\ee
Obviously,  $\lambda_+(\beta, B)> \lambda_-(\beta, B)$. In the case of free boundary conditions, we observe that
	\eqan{
	\label{tre_1}
	Z^{\sss (l)}_{\sss N} 
	&=  \sum_{\sigma_1=\pm 1}\cdots\sum_{\sigma_{\sss N}=\pm 1}
	\e^{\frac{B}{2}\sigma_1}D_{\sigma_1,\sigma_2}D_{\sigma_2,\sigma_3}
	\cdots D_{\sigma_{N-1,}\sigma_{\sss N}}\e^{\frac{B}{2}\sigma_{\sss N}} \\
	&=  \sum_{\sigma_1=\pm 1}\sum_{\sigma_{\sss N}=\pm 1}  \e^{\frac{B}{2}\sigma_1}
	D^{N-1}_{\sigma_1,\sigma_{\sss N}}\e^{\frac{B}{2}\sigma_{\sss N}} \, = \, v^T D^{N-1} v,\nn
	}
where the vector $v$ is defined by $v^{T} = (\e^{B/2},\e^{-B/2})$. This can be written as
	\be
	\label{cinque5}
	Z^{\sss (l)}_{\sss N}  =  A_+ \lambda_+^N + A_- \lambda_-^N,
	\qquad \mbox{with}\qquad  A_{\pm} = \frac{(v^T \cdot v_{\pm})^2}{\lambda_{\pm}},
	\ee
where $v_{\pm}$ denote the two orthonormal eigenvector of the matrix $D$, and therefore
	$$
	A_{\pm} = A_{\pm}(\beta, B) 
	= \frac{\e^{-2\beta}\e^{\pm B} + (\lambda_+-\e^{\beta+B})^2 \e^{\mp B} 
	\pm 2 \e^{-\beta} (\lambda_+ - \e^{\beta+B})}
	{[\e^{-2\beta} + (\lambda_+ - \e^{\beta+B})^2 ] \lambda_{\pm}}.
	$$
Let us remark that, since $\beta > 0$, $\lambda_{\pm}(\beta, B)>0$ and  $A_{\pm} \left(\beta, B\right) > 0.$
From \eqref{tre} and \eqref{cinque5}, it follows that the pressure of the one-dimensional Ising model is,
independently of the boundary conditions, given by
	\be
	\psi^{d=1}(\beta, B)=\log \lambda_+(\beta, B) \, .
	\ee

\subsection{Quenched results: Proof of Theorem \ref{CLT_an1}}
In this section, we prove Theorem \ref{CLT_an1}. We start by analyzing the quenched pressure.

\subsubsection{Quenched pressure}
It is not difficult to see that  any 2-regular random graph is formed by cycles only.  Thus,  denoting by   $K^t_{\sss N}$ the  random number of cycle \col{in} the graph, we can enumerate them in an arbitrary order from 1 to $K^t_{\sss N}$  and  call $L_{\sss N}(i)$ the length (i.e. the number of vertices)  of the $i$-th cycle.  The random variable $K^t_{\sss N}$  is 
given  by 
	\be
	\label{kappa_N}
	K^t_{\sss N}= \sum_{j=1}^N I_j,
	\ee
where $I_j$  are independent Bernoulli variables, i.e.,
	\be\label{bernI}
	I_j = \mbox{Bern}\,\left(\frac{1}{2N -2j+ 1}\right).
	\ee
Indeed,  at every step $j$ in the construction of the 2-regular random graph, namely at each pairing of two half-edges, we have one and only one possibility to close a cycle. This possibility corresponds to drawing exactly one half-edge out of  the remaining $2N-2j+1$ unpaired half-edges. The indicator of this event  is the Bernoulli the variable  $I_j$. Obviously, the variables $\left(I_j\right)_{j=1}^N$ are independent, but not identical, and the number of the cycles $K^t_{\sss N}$ is distributed as their sum.

Since the random graph splits into (disjoint) cycles, its partition function factorizes into the product of the partition functions of each cycle. Therefore,
	\be
	\label{zprod}
	Z_{\sss N}(\beta, B) = \prod_{i=1}^{K^t_{\sss N}} Z_{L_{\sss N}(i)}^{\sss (t)} (\beta, B),
	\ee
where, by (\ref{tre}) \col{the partition function of the $i^{th}$ cycle is}
	\be
	Z_{L_{\sss N}(i)}^{\sss (t)}(\beta, B) 
	= \lambda_+^{L_{\sss N}(i)}(\beta, B) + \lambda_-^{L_{\sss N}(i)}(\beta, B).
	\ee
Because $\beta > 0$, we have $0 < \lambda_-(\beta, B) < \lambda_+(\beta, B)$, so that, for every $i$,
	\be
	\lambda_+^{L_{\sss N}(i)}(\beta, B) \leq Z_{L_{\sss N}(i)}^{\sss (t)}(\beta, B) 
	\leq 2 \lambda_+^{L_{\sss N}(i)}(\beta, B).
	\ee
As a result, we can bound the the pressure by
	\be
	\prod_{i=1}^{K^t_{\sss N}} \lambda_+^{L_{\sss N}(i)}(\beta, B) 
	\leq \prod_{i=1}^{K^t_{\sss N}} Z_{L_{\sss N}(i)}^{\sss (t)}(\beta, B) 
	\leq \prod_{i=1}^{K^t_{\sss N}} 2 \lambda_+^{L_{\sss N}(i)}(\beta, B),
	\ee
and, since $\sum_{i=1}^{K^t_{\sss N}}L_{\sss N}(i) = N$, we finally obtain
	\be\label{Bound}
	\lambda_+^N(\beta, B) \leq Z_{\sss N}(\beta, B) \leq 2^{K^t_{\sss N}} \lambda_+^N(\beta, B).
	\ee
Now we are ready to compute the quenched pressure in the thermodynamic limit, defined as in \eqref{lim_press}. By the previous inequality,
	\be
	\label{ineqZCM2}
	\log \lambda_+(\beta, B) \leq \psi_{\sss N} (\beta, B) 
	\leq \log 2 \cdot \frac{K^t_{\sss N}}{N}  + \log \lambda_+(\beta, B),
	\ee
where $ \psi_{\sss N} (\beta, B)$ is the random quenched pressure defined in \eqref{press_N}.  Now, by \eqref{kappa_N} and \eqref{bernI},
	\be
	\label{lim_av_k}
	\frac{1}{N} \q \left({K^t_{\sss N}} \right) 
	= \frac{1}{N}\sum_{i=1}^N\frac{1}{2N - 2i +1} 
	\, \sim \, \frac{\log N}{N} \to 0\quad \mbox{as}\quad N\to \infty.
	\ee
Therefore, by applying Markov's inequality to the non-negative variable $\frac{K^t_{\sss N}}{N} $,
	\be
	\label{kappaNsuN}
	\frac{K^t_{\sss N}}{N} \stackrel{\mathbb{P}}{\longrightarrow} 0, \qquad \mbox{w.r.t.} \; \, \q\, .
	\ee
Hence,  taking the limits in \eqref{ineqZCM2},  we obtain the infinite volume random quenched pressure:
	\be
	\psi (\beta, B) = \psi^{d=1} (\beta, B)\equiv \log \lambda_+(\beta, B).
	\ee
\col{Moreover, from Corollary \ref{corol_avquetermlim} we also obtain the averaged quenched pressure}	\be
	\label{psiavqueCM2}
	\overline{\psi} (\beta, B) = \psi (\beta, B) =  \psi^{d=1} (\beta, B).
	\ee

\subsubsection{Cumulant generating functions}
As already said in Section \ref{sec_proof_strategy}, in order to prove the averaged quenched CLT for  $\CMNtwo$, we need to calculate the limit for $N \rightarrow \infty$ of the cumulant generating function in the averaged quenched random setting. We write this function 
in the following form:
	\be\label{cbarra}
	\overline{c}_{\sss N}(t) \,=\, \frac{1}{N} \log \p \left[\exp(tS_{\sss N})\right]\, 
	= \, \frac{1}{N} \log \q \left[\frac{Z^{\sss (t)}_{\sss N} \left( \beta, B+t \right)}
	{Z^{\sss (t)}_{\sss N} \left( \beta, B \right)}\right],
	\ee
where, in the right-hand side of the previous equation, we have expressed the random quenched average of $\exp(tS_{\sss N})$ as a ratio of partition functions, i.e.,  $\m[\exp(tS_{\sss N})]=Z^{\sss (t)}_{\sss N} \left( \beta, B+t \right)/Z^{\sss (t)}_{\sss N} \left( \beta, B \right)$.  Again using (\ref{Bound}), we  bound the ratio of the partition functions at different fields as
	\be
	2^{-K^t_{\sss N}}\left(\frac{\lambda_+\left(\beta, B+t \right)}{\lambda_+\left(\beta, B \right)}\right)^N 	
	\leq \frac{Z_{\sss N} \left( \beta, B+t \right)}{Z_{\sss N} \left( \beta, B \right)} 
	\leq 2^{K^t_{\sss N}}\left(\frac{\lambda_+\left(\beta, B+t \right)}
	{\lambda_+\left(\beta, B \right)}\right)^N.
	\ee
Therefore,
	\be
	\frac{1}{N} \log \q 
	\left(2^{-K^t_{\sss N}}\left(\frac{\lambda_+\left(\beta, B+t \right)}
	{\lambda_+\left(\beta, B \right)}\right)^N\right)
	\leq  \overline{c}_{\sss N}(t)  \leq \frac{1}{N} \log\q \left(2^{K^t_{\sss N}}
	\left(\frac{\lambda_+\left(\beta, B+t \right)}{\lambda_+\left(\beta, B \right)}\right)^N\right),
	\ee
and, recalling that $ \lambda_+(\beta, B) $ and $ \lambda_+(\beta, B+t) $ are not random, we obtain
	\be
	\frac{1}{N} \log \left(\q \left(2^{-K^t_{\sss N}}\right)\right)
	 \leq \overline{c}_{\sss N}(t) - \log \lambda_+(\beta, B+t) +\log \lambda_+(\beta, B)  
	\leq \frac{1}{N} \log \left(\q \left(2^{K^t_{\sss N}}\right)\right).
	\ee
Now we can compute the limit of the averages in the previous equation recalling that $K^t_{\sss N}$ is a sum of independent Bernoulli variables $I_{i}$, c.f.  \eqref{kappa_N}. Indeed,
	\begin{align}
	\label{dueallaK}\nonumber
	\frac{1}{N} \log \left(\q \left(2^{K^t_{\sss N}}\right)\right) 
	&= \frac{1}{N} \log \prod_{i=1}^N \q \left(2^{I_i}\right)
	\, = \, \frac{1}{N} \log 
	\prod_{i=1}^N \left[ \frac{2}{2N -2i +1} + \left(1 - \frac{1}{2N -2i +1} \right)\right] & \\ 
	& = \, \frac{1}{N} \sum_{i=1}^N \log \left(1 + \frac{1}{2N -2i +1} \right) 
	\stackrel{N \rightarrow \infty}{\longrightarrow} 0,
	\end{align}
where the limit can be obtained using the inequality $\log(1+x)\le x$ for $x\ge 0$, and the same asymptotic estimate  already used in \eqref{lim_av_k}. In a similar fashion we can also prove that
	\be\label{dueallaKbis}
	\lim_{N \rightarrow \infty} \, \frac{1}{N} \log \left(\q \left(2^{-K^t_{\sss N}}\right)\right) = 0.
	\ee
By  \eqref{dueallaK} and  \eqref{dueallaKbis} we finally obtain the limit of the cumulant generating function in the averaged quenched setting:
	\be\label{stessec}
	\overline{c}(t) = \lim_{N \rightarrow \infty} \overline{c}_{\sss N}(t)
	= \psi (\beta, B+t) - \psi (\beta, B) 
	= \log \lambda_+(\beta, B+t) - \log \lambda_+(\beta, B)\;.
	\ee
On the other hand, since by \eqref{c_infinito},  we have also that $c(t) \,= \, \psi (\beta, B+t) - \psi (\beta, B)$, we conclude that in the thermodynamic limit the averaged quenched and the random quenched cumulant generating functions of the total spin are the same:
	\be\label{stessec2}
	c(t)=\overline{c}(t) = \log \lambda_+(\beta, B+t) - \log \lambda_+(\beta, B).
	\ee

\subsubsection{Averaged quenched CLT: proof of Theorem \ref{CLT_an1}}
\label{prf_aq_CLT}
According to our general strategy, in order to prove Theorem \ref{CLT_an1}, we need to show that
	\be
	\lim_{N \rightarrow \infty} 
	\p \left[\exp\left(t \, \frac{S_{\sss N} - P_{\sss N}(S_{\sss N})}{\sqrt{N}}\right) 	\right] 
	\, = \, \exp \left(\frac{1}{2} \sigma_{\mathrm{aq}}^2(\beta, B) t^2\right)  
	\qquad \mbox{for all}\; t \in \mathbb{R}.
	\ee
As in the proof of Theorem \ref{random-CLT}, the limit can be computed by expressing the expectation on the left-hand side in terms of the proper generating function. Here, this is the averaged quenched cumulant generating function, i.e., $\overline{c}_{\sss N}(t)$. Indeed, using the fact that $\overline{c}_{\sss N}(0) = 0$ \col{and $\overline{c}'_{\sss N}(0) = P_{\sss N}(\frac{S_{\sss N}}{N})$},  see \eqref{cbarra}, by Taylor's theorem with Lagrange remainder,
	\begin{align}\nonumber\label{popi}
	\log \p  \left[\exp\left( \frac{t S_{\sss N} - t P_{\sss N}(S_{\sss N})}{\sqrt{N}}\right) \right] \, 
	&= \, \log \left[\p\left(\exp \left(\frac{t}{\sqrt{N}} S_{\sss N}\right)\right)\right] 
	- \frac{t}{\sqrt{N}} \p(S_{\sss N}) \\
	& = \, N \overline{c}_{\sss N} \left(\frac{t}{\sqrt{N}}\right) - t \sqrt{N} \overline{c}'_{\sss N} (0) 
	\: = \: \frac{t^2}{2} \overline{c}''_{\sss N}(t_{\sss N}),
	\end{align}
for some $t_{\sss N} \in [0 , \frac{t}{\sqrt{N}}]$.  In order to compute the limit of the sequence $\overline{c}''_{\sss N}(t_{\sss N})$, we consider
	\begin{align}\label{c_differenza}\nonumber
	\overline{c}_{\sss N}(t) - \overline{c}(t) 
	& = \frac{1}{N} \log \q \left[\frac{Z^{\sss (t)}_{\sss N}(\beta, B+t)}{Z^{\sss (t)}_{\sss N}(\beta, B)}\right] 
	- \log \left[\frac{\lambda_+(\beta, B+t)}{\lambda_+(\beta, B)}\right]
	= \frac{1}{N} \log \q  \left[ \frac{\frac{Z^{\sss (t)}_{\sss N}(\beta, B+t)}
	{\left(\lambda_+(\beta, B+t)\right)^N}}
	{\frac{Z^{\sss (t)}_{\sss N}(\beta, B)}{\left(\lambda_+(\beta, B)\right)^N}}\right]\\
	& = \frac{1}{N} \log \q \left[\prod_{i=1}^{K^t_{\sss N}} 
	\frac{1 + \left(r_{\sss B+t}\right)^{L_{\sss N}(i)}}{1 + \left(r_{\sss B}\right)^{L_{\sss N}(i)}} \right],
	\end{align}
where we have used \eqref{tre}, \eqref{zprod}, \eqref{cbarra}, \eqref{stessec}, and, omitting the dependence on $\beta$, we have defined
	\be\label{defalpha}
	r_{\sss B} \, = \, r(\beta, B) \, = \, \frac{\lambda_-(\beta, B)}{\lambda_+(\beta, B)}\;. 
	\ee
Then, recalling \eqref{stessec}, we can rewrite
	\be
	\label{c_diff}
	\overline{c}_{\sss N}(t) 
	= \log \lambda_+(\beta, B+t) - \log \lambda_+(\beta, B) 
	+ \frac{1}{N} \log \q \left[\prod_{i=1}^{K^t_{\sss N}} \frac{1 + \left(r_{\sss B+t}\right)^{L_{\sss N}(i)}}{1 
	+ \left(r_{\sss B}\right)^{L_{\sss N}(i)}} \right]. 
	\ee
The second derivative of \eqref{c_diff} is
	\be
	\label{cbarra2}
	\overline{c}''_{\sss N}(t) \,=\, \frac{\partial^2}{\partial t^2} \log \lambda_+(\beta, B+t) 
	\,+\, \frac{1}{ND_{\sss N}(t)}\left[I_{\sss N}(t) + II_{\sss N}(t) + \frac{III_{\sss N}(t)}{D_{\sss N}(t)}\right],
	\ee
where we have introduced
	\eqan{
	I_{\sss N} (t)&= \q\Big[ \displaystyle \sum_{i=1}^{K^t_{\sss N}} 
	\frac{L_{\sss N}(i)(r_{\sss B+t})^{\sss L_{\sss N}(i)-2}[(L_{\sss N}(i)-1)(r'_{\sss B+t})^2  
	+ r_{\sss B+t}r''_{\sss B+t}]}
	{1 + (r_{\sss B})^{\sss L_{\sss N}(i)}} 
	\prod \limits_{\substack{j=1 \\ j\neq i}}^{K^t_{\sss N}}\frac{1 + (r_{\sss B+t})^{\sss L_{\sss N}(j)}}
	{1 + (r_{\sss B})^{\sss L_{\sss N}(j)}} \Big],\\
	II_{\sss N}(t)  &=  \q \Big[ \displaystyle \sum_{i=1}^{K^t_{\sss N}}
	\sum \limits_{\substack{j=1 \\ j\neq i}}^{K^t_{\sss N}} 
	\frac{L_{\sss N}(i)L_{\sss N}(j) (r_{\sss B+t})^{\sss L_{\sss N}(i)+L_{\sss N}(j)-2}(r'_{\sss B+t})^2}
	{\left(1 + (r_{\sss B})^{\sss L_{\sss N}(i)}\right)\left(1 + (r_{\sss B})^{\sss L_{\sss N}(j)}\right)} 
	\prod \limits_{\substack{l=1 \\ l\neq i,j}}^{K^t_{\sss N}}
	\frac{1 + (r_{\sss B+t})^{\sss L_{\sss N}(l)}}{1 + (r_{\sss B})^{\sss L_{\sss N}(l)}} \Big],\\
	III_{\sss N}(t) &=  \Big[\q \Big(\displaystyle 
	\sum_{i=1}^{K^t_{\sss N}}\frac{L_{\sss N}(i)(r_{\sss B+t})^{\sss L_{\sss N}(i)-1} r'_{\sss B+t}}
	{1 + (r_{\sss B})^{\sss L_{\sss N}(i)}}
	\prod \limits_{\substack{j=1 \\ j\neq i}}^{K^t_{\sss N}}\frac{1 + (r_{\sss B+t})^{\sss L_{\sss N}(j)}}
	{1 + (r_{\sss B})^{\sss L_{\sss N}(j)}}\Big)\Big]^2,\\
	D_{\sss N} (t) &=  \q\Big[\displaystyle \prod_{i=1}^{K^t_{\sss N}}
	\frac{1 + (r_{\sss B+t})^{\sss L_{\sss N}(i)}}{1 + (r_{\sss B})^{\sss L_{\sss N}(i)}}\Big].
	}

All these terms are of a similar structure. The final step in the proof of the theorem requires to compute the limit of the second term in the right-hand-side of \eqref{cbarra2}. This step is taken in the following lemma:
\begin{lemma}[]
\label{limc2}
Let $t>0$ and let $(t_{\sss N})_{N\geq 1}$ be a sequence of real numbers such that $t_{\sss N} \in [0 , t/\sqrt{N}]$. Then,
	\be
	\lim_{N \rightarrow \infty} 
	\label{necessary-bd}
	\frac{1}{ND_{\sss N}(t_{\sss N})}\left[I_{\sss N}(t_{\sss N}) + II_{\sss N}(t_{\sss N}) + 
	\frac{III_{\sss N}(t_{\sss N})}{D_{\sss N}(t_{\sss N})}\right] = 0.
	\ee
\end{lemma}

\begin{proof}
We first consider the term $I_{\sss N}(t_{\sss N})$. From \eqref{defalpha}, it is clear that $0<r(\beta,B)<1$. As a consequence, the terms $L_{\sss N}(i)(L_{\sss N}(i)-1)(r_{\sss B+t_{\sss N}})^{\sss L_{\sss N}(i)-2}$ and $ L_{\sss N}(i)(r_{\sss B+t_{\sss N}})^{\sss L_{\sss N}(i)-1}$, appearing in the numerator of $I_{\sss N}(t)$ are uniformly bounded in $N$.  Moreover, also the sequences  $(r'_{\sss B+t_{\sss N}})_{\sss N}$ and $ (r''_{\sss B+t_{\sss N}})_{\sss N}$, being convergent, are bounded. Therefore, there exists a constant $C>0$ such that
	\be
	\left|L_{\sss N}(i)(L_{\sss N}(i)-1)(r_{\sss B+t_{\sss N}})^{\sss L_{\sss N}(i)-2}(r'_{\sss B+t_{\sss N}})^2  
	+ L_{\sss N}(i)(r_{\sss B+t_{\sss N}})^{\sss L_{\sss N}(i)-1} r''_{\sss B+t_{\sss N}}\right|  \, \leq \, C.
	\ee
Therefore, also
	\be
	\sum_{i=1}^{K^t_{\sss N}}\frac{\left|L_{\sss N}(i)(L_{\sss N}(i)-1)
	(r_{\sss B+t_{\sss N}})^{\sss L_{\sss N}(i)-2}(r'_{\sss B+t_{\sss N}})^2  
	+ L_{\sss N}(i)(r_{\sss B+t_{\sss N}})^{\sss L_{\sss N}(i)-1} 
	r''_{\sss B+t_{\sss N}}\right| }{1 + (r_{\sss B})^{\sss L_{\sss N}(i)}}\, 
	\leq \,C \cdot K^t_{\sss N}, 
	\ee
and thus
	\be
	\left|I_{\sss N}(t_{\sss N})\right| \, \leq \, C \cdot \q \left(K^t_{\sss N} \, 
	\prod_{i=1}^{K^t_{\sss N}}\frac{1 + (r_{\sss B+t_{\sss N}})^{\sss L_{\sss N}(i)}}
	{1 + (r_{\sss B})^{\sss L_{\sss N}(i)}}\right).
	\ee
We now proceed by estimating first each term appearing in the product. Let us fix an arbitrary length $l\in \mathbb{N}$ to be chosen later on. Recalling that $0<r_{\sss B}<1$ we obtain
	\begin{equation} 
	\frac{1 + (r_{\sss B+t_{\sss N}})^{\sss L_{\sss N}(j)}}{1 + (r_{\sss B})^{\sss L_{\sss N}(j)}}\, 
	\leq\, \left\{\begin{array}{ll} 1 + (r_{\sss B+t_{\sss N}})^{\sss l}  \quad \quad \; \; \; 
	\mbox{if}\; L_{\sss N}(j) > l,\\ 1+ \delta_{\sss N}(l) \quad \quad \quad \quad  
	\mbox{if}\; L_{\sss N}(j) \leq l, \end{array}\right.  
	\end{equation}
where $\displaystyle \delta_{\sss N}(l) := \max_{k \leq l} \left[\frac{1 + (r_{\sss B+t_{\sss N}})^{\sss k}}{1 + (r_{\sss B})^{\sss k}} -1\right] \, \stackrel{N \rightarrow \infty }{\longrightarrow}\, 0$. Then,
	\be\label{stima_esp}
	\frac{1 + (r_{\sss B+t_{\sss N}})^{\sss L_{\sss N}(j)}}{1 + (r_{\sss B})^{\sss L_{\sss N}(j)}}\, 
	\leq\, 1 + \left((r_{\sss B+t_{\sss N}})^{\sss l} \vee \delta_{\sss N}(l) \right) \, 
	\leq\, \e^{\left( (r_{\sss B+t_{\sss N}})^{\sss l} \vee \delta_{\sss N}(l) \right)},
	\ee
where $a \vee b=\max \{ a,b\}$. Finally, using (\ref{stima_esp}) and the Cauchy-Schwarz inequality
	\begin{align}\label{stima_I_N}
	\left|I_{\sss N}(t_{\sss N})\right| \, & \leq C \cdot \q \Big(K^t_{\sss N} \, 
	\prod \limits_{i=1}^{K^t_{\sss N}} \frac{1 + (r_{\sss B+t_{\sss N}})^{\sss L_{\sss N}(i)}}
	{1 + (r_{\sss B})^{\sss L_{\sss N}(i)}}\Big) 
	\leq C \cdot  \q \left(K^t_{\sss N} \, \e^{\sum_{i=1}^{K^t_{\sss N}} 
	\left((r_{\sss B+t_{\sss N}})^{\sss l} \vee \delta_{\sss N}(l) \right)}\right) \, \\ \nonumber
	&\leq C \cdot  \q \left(K^t_{\sss N} \, \e^{K^t_{\sss N} 
	\left( (r_{\sss B+t_{\sss N}})^{\sss l} \vee \delta_{\sss N}(l) \right)}\right) 
	 \leq C \cdot \left[\q \left(\left(K^t_{\sss N}\right)^2\right)\right]^{1/2} 
	\left[\q\left(\e^{2K^t_{\sss N} 
	\left((r_{\sss B+t_{\sss N}})^{\sss l} \vee \delta_{\sss N}(l) \right)}\right) \right]^{1/2}.
	\end{align}
We next consider $D_{\sss N}(t_{\sss N})$. As before, 
	\begin{equation} 
	\frac{1 + (r_{\sss B+t_{\sss N}})^{\sss L_{\sss N}(j)}}{1 + (r_{\sss B})^{\sss L_{\sss N}(j)}}\, \geq\, 	
	\left\{\begin{array}{ll} 
	\displaystyle (1 + (r_{\sss B})^{\sss l})^{-1}\quad \quad \quad  \; 
	\mbox{if}\; L_{\sss N}(j) > \, l \, ,\\  (1+ \bar{\delta}_{\sss N}(l))^{-1} \quad \quad \quad \; 
	\mbox{if}\; L_{\sss N}(j) \leq  \, l\, , \end{array}\right.  
	\end{equation}
where now $\bar{\delta}_{\sss N}(l) \, := \, \min_{k \leq l} \left[ \frac{1 + (r_{\sss B})^{\sss k}}{1 + (r_{\sss B+t_{\sss N}})^{\sss k}} -1  \right] \stackrel{N \rightarrow \infty }{\longrightarrow}\, 0.$
Then,
	\be
	\frac{1 + (r_{\sss B+t_{\sss N}})^{\sss L_{\sss N}(j)}}{1 + (r_{\sss B})^{\sss L_{\sss N}(j)}}\, 
	\geq\, \frac{1}{1 + \left((r_{\sss B})^{\sss l} \vee \bar{\delta}_{\sss N}(l) \right)} \, 
	\geq \, \e^{- \left((r_{\sss B})^{\sss l} \vee \bar{\delta}_{\sss N}(l) \right)},
	\ee
because $\, \frac{1}{1+x} \geq \e^{-x} \;$ for $\,x>-1$. From the previous bound, we obtain
	\be\label{stima_D_N}
	D_{\sss N} (t_{\sss N})\, \geq \, \q\Big(\displaystyle \prod_{i=1}^{K^t_{\sss N}} 
	\e^{- \left((r_{\sss B})^{\sss l} \vee \bar{\delta}_{\sss N}(l) \right)}  \Big) \, 
	=  \, \q \left(\e^{- \sum_{i=1}^{K^t_{\sss N}} \left((r_{\sss B})^{\sss l} \vee \bar{\delta}_{\sss N}(l) \right)}\right) \,  
	= \, \q \left(\e^{- K^t_{\sss N} \left((r_{\sss B})^{\sss l} \vee \bar{\delta}_{\sss N}(l) \right)}\right).
	\ee
Collecting (\ref{stima_I_N}) and (\ref{stima_D_N}) we obtain
	\begin{align}\label{maggCM2asimpt}
	\lim_{N \rightarrow \infty} \, \frac{\left|I_{\sss N}(t_{\sss N})\right|}{N D_{\sss N}(t_{\sss N})} \, 
	&\leq \,  \lim_{N \rightarrow \infty} \, \frac{C \cdot \left[\q \left(\left(K^t_{\sss N}\right)^2\right)\right]^{1/2}
	\left[\q\left(\e^{2K^t_{\sss N} \left((r_{\sss B+t_{\sss N}})^{\sss l} \vee \delta_{\sss N}(l) \right)}\right) \right]^{1/2}}
	{N \cdot \q \left(\e^{- K^t_{\sss N} \left((r_{\sss B})^{\sss l} \vee \bar{\delta}_{\sss N}(l) \right)}\right)}. 
	\end{align}
The averages in the right-hand side can be bounded using the distribution of $K^{t}_{\sss N}$ in \eqref{kappa_N} and \eqref{bernI}. Indeed, using that, for fixed $l$, $\delta_{\sss N}(l)\to 0$ and that $(r_{\sss B+t_{\sss N}})^{\sss l}$ can be made small by taking $l$ large, 
we can make $(r_{\sss B})^{\sss l} \vee \bar{\delta}_{\sss N}(l)\leq \vep$ for $N$ large enough. Thus,
	\be
	\q\left(\e^{2K^t_{\sss N} \left((r_{\sss B+t_{\sss N}})^{\sss l} \vee \delta_{\sss N}(l) \right)}\right)
	\leq \prod_{i=1}^N\Big(1 + \frac{\left(\e^{\vep} -1\right)}{2(N -i) + 1}\Big)
	=\frac{\Gamma \left (N +\frac{1}{2}\e^{\vep} \right)\sqrt{\pi}}
	{\Gamma(N +\frac{1}{2})\Gamma(\frac 1 2\e^{\vep})}\sim N^{\frac{1}{2}(\e^{\vep} -1)},
	\ee
where the product is expressed in terms of the Gamma function, and the asymptotic relation $\frac{\Gamma(N+a)}{\Gamma(N+b)}=N^{a-b}(1+o(1))$, valid when $a,b$ are uniformly bounded, has been used.


In a similar fashion the average in the denominator of  \eqref{maggCM2asimpt} can be bounded by
	\be
	\q \left(\e^{- K^t_{\sss N} \left((r_{\sss B})^{\sss l} \vee \bar{\delta}_{\sss N}(l) \right)}\right) 
	\geq N^{\frac{1}{2}(\e^{ - \vep} -1)},
	\ee
while $\left[\q \left(\left(K^t_{\sss N}\right)^2\right)\right]^{1/2}=O(\log N)$. Thus,  we conclude that the right-hand side of \eqref{maggCM2asimpt} is bounded by $O(\log N) N^{\frac{3}{4}(\e^{\vep}-1)-1}=o(1)$ by taking $\vep>0$ to be sufficiently small. This proves the bound on $I_{\sss N}(t_{\sss N})$ in \eqref{necessary-bd}.
With similar calculations, that we omit for the sake of brevity,  we can also show that the terms involving $II_{\sss N}(t_{\sss N})$ and $III_{\sss N}(t_{\sss N})$ vanish, thus
concluding the proof of  Lemma \ref{limc2}.
\end{proof}

\noindent
{\it  Proof of Theorem \ref{CLT_an1}.} Equations \eqref{popi}, \eqref{cbarra2} and Lemma \ref{limc2} complete the proof of the averaged quenched CLT, since
	\begin{align}
	\lim_{N \rightarrow \infty} \log \p \left[\exp\left(t \, \frac{S_{\sss N} - P_{\sss N}(S_{\sss N})}{\sqrt{N}}\right) \right] \, 
	&= \, \lim_{N \rightarrow \infty} \frac{t^2}{2}\overline{c}''_{\sss N}(t_{\sss N}) \,\nonumber 
	= \, \frac{t^2}{2}\cdot \left.\frac{\partial^2}{\partial t^2} \log \lambda_+(\beta, B+t)\right|_{t=0}\\
	&= \, \frac{t^2}{2}\cdot \frac{\cosh(B) \e^{-4\beta}}{(\sinh(B)+\e^{-4\beta})^{3/2}}.
	\end{align}
Since the thermodynamic limits of the two cumulant generating functions $c(t)$ and $\overline{c}(t)$ are the same, cf. (\ref{stessec2}), also the variances of the limiting normal distribution in the random quenched and the averaged quenched are equal.
\qed

\section{Proofs for $\CMNonetwo$}
\label{four}
In this section, we consider the configuration model $\CMNonetwo$, introduced in Section 1.2. In this graph, the connected components are either  tori $(t)$ connecting vertices of degree $2$, or lines $(l)$ formed joining vertices with degree 2 and ending with two vertices with degree 1. In order to state some properties of the number of lines and tori, we need to introduce some notation. Recall that the number of vertices of degree 1 and 2 are given by
	\begin{align}
	\label{defn1n2}
	n_1 &\;:=\;  \# \left\{i \in [N] \colon d_i = 1\right\}\,=\, N - \lfloor pN \rfloor,\qquad
	n_2  \;:=\;  \# \, \left\{i \in [N] \colon d_i =2\right\}\,=\, \lfloor pN\rfloor,
	\end{align}
and the total degree of the graph
	\be
	\label{defln}
	\ell_{\sss N} \; =\; \sum_{i \in [N]} d_{i} \,=\, 2n_2 + n_1 \,=\, N+\lfloor pN\rfloor.
	\ee
The number of edges is given by $\ell_{\sss N}/2$, so we assume $n_1$ to be even.  Let us also denote by $K_{\sss N}$ the number of connected components in the graph and by $K_{\sss N}^{\sss (l)}$ and $K_{\sss N}^{\sss (t)}$ the number lines and tori. Obviously,
	\be
	K_{\sss N}= K_{\sss N}^{\sss (l)} + K_{\sss N}^{\sss (t)}.
	\ee
Because every line uses up two vertices of degree 1, the number of lines is given by $n_1/2$, i.e., $K_{\sss N}^{\sss (l)} = (N - \lfloor pN \rfloor)/2,$ almost surely. Regarding the number of cycles, we have that $K_{\sss N}^{\sss (t)}/N$ has the same distribution  of $K_{\bar{N}}^t/N$, where $\overline{N}$ is the (random) number of vertices with degree 2 that do not belong to any line and $K_{\bar{N}}^t$ is the number of tori on this set of vertices.  Then, since this subset forms a $\mathrm{CM}_{\sss \bar{N}}(\mathrm{\textbf{2}})$ graph, we can apply  (\ref{kappaNsuN}), obtaining that $K_{\sss N}^{\sss (t)}/N\stackrel{\mathbb{P}}{\longrightarrow} 0$, so that also
	\be
	\frac{K_{\sss N}}{N} \stackrel{\mathbb{P}}{\longrightarrow} \frac{(1-p)}{2}.
	\ee
The partition function can be computed as in the case of $\CMNtwo$, and is given by
	\be
	\label{zcm12}
	Z_{\sss N}(\beta, B) = \prod_{i=1}^{K_{\sss N}^{\sss (l)}} Z_{L^{\sss (l)}_{\sss N}(i)}^{\sss (l)} (\beta, B) \cdot 	
	\prod_{i=1}^{K_{\sss N}^{\sss (t)}} Z_{L^{\sss (t)}_{\sss N}(i)}^{\sss (t)} (\beta, B),
	\ee
where by (\ref{cinque5}) and (\ref{tre})
	\be
	Z_{L^{\sss (l)}_{\sss N}(i)}^{\sss (l)} (\beta, B) 
	= \Big(A_+ \lambda_+^{L_{\sss N}^{\sss (l)}(i)} + A_- \lambda_-^{L_{\sss N}^{\sss (l)}(i)}\Big),\qquad 
	Z_{L^{\sss (t)}_{\sss N}(i)}^{\sss (t)} (\beta, B) 
	= \Big(\lambda_+^{L_{\sss N}^{\sss (t)}(i)} + \lambda_-^{L_{\sss N}^{\sss (t)}(i)}\Big),
	\ee
and $L^{\sss (l)}_{\sss N}(i)$ and $L^{\sss (t)}_{\sss N}(j)$ are the lengths (i.e., the number of vertices) of the $i$th line and $j$th torus (in any arbitrary labeling).

\subsection{Quenched results: Proof of Theorem \ref{CLT_an2}}

\subsubsection{Quenched pressure}
Starting from  the expression of the partition function \eqref{zcm12} and with some algebraic manipulations, we obtain the random quenched pressure at volume $N$ as
	\eqan{\label{tori_linee}
	\psi_{\sss N}(\beta,B) 
	\, = \, &\log \lambda_+(\beta,B) + \sum_{l=2}^{N} p_l^{\sss{(N)}} 	
	\log \left(A_+(\beta,B) + A_-(\beta,B) \left(r_{\sss B}\right)^{l}\right)\\
	& + \frac{1}{N} \sum_{i=1}^{K_{\sss N}^{\sss (t)}} \log \left(1 + \left(r_{\sss B}\right)^{L_{\sss N}^{\sss (t)}(i)} \right),\nn
	}
where the first sum is taken over the possible line lengths and we define
	\be\label{p_elle}
	p_l^{\sss{(N)}} \,:=\,  \frac{1}{N} \sum_{i=1}^{K_{\sss N}^{\sss (l)}} 
	\mathbbm{1}_{\{L_{\sss N}^{\sss (l)}(i) = l\}}
	\ee
to be the normalized number of lines of length $l$ in the $\CMNonetwo$ graph.  The random variables $(p_l^{\sss{(N)}})_{l\geq 2}$ play a fundamental role in the proofs, therefore we start by investigating their behavior in the thermodynamic limit:

\begin{lemma}[Convergence of lines of given lengths]
\label{p_l_1} 
For every \,$ l \geq 2 $,\, as \,$N \rightarrow \infty$,
	\be
	p_l^{\sss{(N)}} \stackrel{\mathbb{P}}{\longrightarrow} p_l^* 
	\,= \, \left(\frac{2p}{p + 1}\right)^{l-2} \left(\frac{1-p}{p +1}\right) \left(\frac{1-p}{2}\right).
	\ee
\end{lemma}

\begin{proof} The result can be obtained by applying the Second Moment Method, which states that the convergence of the averages $ \q \left(p_l^{\sss{(N)}}\right) \to \, p_l^* $ and the vanishing of the variances $\mathrm{Var}_{\sss \q}\left(p_l^{\sss{(N)}}\right){\to} 0$ together imply the required convergence in probability. We start by computing the average of  $p^{\sss(N)}_l$. From \eqref{p_elle}, we obtain
	\be
	\label{avplL}
	\q (p_l^{\sss{(N)}}) \,=\, \frac{\q(K_{\sss N} (l))}{N} \, \q\left(L_{\sss N}(1)=l\right)
	\,=\, \frac{(1-p)}{2} \, \q\left(L_{\sss N}(1)=l\right)(1+o(1)),
	\ee
because $K_{\sss N}^{\sss (l)} /N\rightarrow (1-p)/2 \; \, \mbox{a.s.}$, and the distribution of lengths $L^{\sss (l)}_{\sss N}(i)$ of the $i$th line is independent of the label $i$. The average on the right hand side of the previous equation can be computed as
	\begin{align}
	\label{avLl}
	\q\left(L_{\sss N}(1)=l\right)\, 
	& = \, \left(\prod_{i=0}^{l-3}\frac{2n_2-2i}{\ell_{\sss N} -1-2i} \right) \left(\frac{n_1-1}{\ell_{\sss N} -1-2(l-2)}\right).
	\end{align}
Recalling \eqref{defn1n2} we have that $n_2/N\to p$ and $n_1/N\to 1-p$ as $N\to \infty$, therefore, by definition \eqref{defln} of $\ell_N$, we obtain that  $\q\left(L_{\sss N}(1)=l\right)\to \left(\frac{2p}{p+1}\right)^{l-2} \left(\frac{1-p}{p+1}\right)$, concluding the computation of the limit of $\q(p_l^{\sss (N)})$. Also the variance of \eqref{p_elle} can be computed explicitly. Indeed, exploiting the fact that $K_{\sss N}^{\sss (l)}$ is a.s.\ constant, \col{we obtain}
	\begin{align}\label{varpln}
	\mathrm{Var}_{\sss \q} (p_l^{\sss{(N)}}) 
	& = \frac{1}{N^2} K_{\sss N}^{\sss (l)} \q\left(L_{\sss N}(1)=l\right) \left(1- \q\left(L_{\sss N}(1)=l\right)\right) \\ 	
	\nonumber
	&\quad \, + \frac{1}{N^2} K_{\sss N}^{\sss (l)} \left(K_{\sss N}^{\sss (l)} - 1\right) 
	\left(\q\left(L_{\sss N}(1)=L_{\sss N}(2)= l\right) - \q\left(L_{\sss N}(1)=l\right)^2 \right)\\
	\nonumber
	& = \frac{1}{N^2}K_{\sss N}^{\sss (l)} \q\left(L_{\sss N}(1)=l\right) \left(1- \q\left(L_{\sss N}(1)=l\right)\right)\\
	\nonumber
	&\quad \, + \frac{1}{N^2} K_{\sss N}^{\sss (l)} \left(K_{\sss N}^{\sss (l)} - 1\right) 
	\q\left(L_{\sss N}(1)=l\right)\left[ \left.\q\left(L_{\sss N}(2) = l\right| L_{\sss N}(1)=l \right) 
	- \q\left(L_{\sss N}(1)=l\right) \right],
	\end{align}
where conditional probability in the previous equation is given by
	\be
	\label{cond_pr_lungh_linee}
	\q( \left. L_{\sss N}(2) = l \right| L_{\sss N}(1)=l) = \Big(\prod_{i=0}^{l-3}\frac{2n_2-2(l-2) -2i}
	{\ell_{\sss N} -2(l-2)-3-2i} \Big)  \Big(\frac{n_1-3}{\ell_{\sss N} -2(l-2) -3-2(l-2)}\Big). 
	\ee
Again recalling \eqref{defn1n2}  and \eqref{defln}, we see that the conditional probability  $\q\left( \left. L_{\sss N}(2) = l \right| L_{\sss N}(1)=l \right) $ converges to $ \left(\frac{2p}{p+1}\right)^{l-2} \left(\frac{1-p}{p+1}\right)$
as $N\to \infty$, so that $\mathrm{Var}_{\sss \q} (p_l^{\sss{(N)}})\,\stackrel{N \rightarrow \infty}{\longrightarrow}\, 0.$
\end{proof}
\medskip

\noindent
With the help of the previous lemma  we will be able to control the contribution of the lines to  $\psi_{\sss N}(\beta,B)$ in the thermodynamic limit  (second term in the right-hand side of \eqref{tori_linee}).  From the next lemma, we deduce that in (\ref{tori_linee}) the contribution of the tori is negligible when $N \rightarrow \infty$:

\begin{lemma}[Bounding the contribution due to tori]
\label{negli}
Conditionally on $L_{\sss N}^{\sss (l)}(1),\ldots,L_{\sss N}^{\sss (l)}(K_{\sss N}^{\sss (l)})$, the number of tori $K_{\sss N}^{\sss (t)}$ in the $\CMNonetwo$ graph is stochastically smaller than the number of tori $K^t_{\sss N}$ in $\CMNtwo$ (and we write $K_{\sss N}^{\sss (t)} \preceq K^t_{\sss N}$). Therefore,
	\be
	\frac{1}{N} \sum_{i=1}^{K_{\sss N}^{\sss (t)}} \log \left(1 + \left(r_{\sss B}\right)^{L_{\sss N}^{\sss (t)}(i)} \right)
	\stackrel{\mathbb{P}}{\longrightarrow} 0.
	\ee
\end{lemma}

\begin{proof} Since $0<r_{\sss B}=\lambda_-/\lambda_+<1$,
	\be
	0\le \frac{1}{N} \sum_{i=1}^{K_{\sss N}^{\sss (t)}} 
	\log \left(1 + \left(r_{\sss B}\right)^{L_{\sss N}^{\sss (t)}(i)} \right) \, \leq \, \frac{K_{\sss N}^{\sss (t)} \log 2}{N} .
	\ee
Recalling that  $K_{\sss N}^{\sss (t)} \preceq K^t_{\sss N}$ and $\displaystyle \frac{ K^t_{\sss N}}{N}\stackrel{\mathbb{P}}{\longrightarrow} 0$ by (\ref{kappaNsuN}), we complete the proof of the lemma.
\end{proof}

By applying  Lemma \ref{p_l_1} and Lemma \ref{negli},  we finally identify the thermodynamic limit of the random quenched pressure as
	\be\label{pressQcaso2}
	\frac{1}{N} \log Z_{\sss N}(\beta, B) \, \stackrel{\mathbb{P}}{\longrightarrow} 
	\psi(\beta,B)\equiv 
	\log \lambda_+(\beta, B)  + \sum_{l \geq 2} p_l^* \log \left(A_+(\beta, B) + A_-(\beta, B) \left(r_{\sss B}\right)^{l}\right).
	\ee
Moreover, recalling that in the thermodynamic limit the averaged quenched pressure is equal to the random quenched pressure,  
\col{we also have} $\overline{\psi} (\beta, B)  \,=\, \psi (\beta, B).$

The quenched susceptibility is then obtained by calculating the second derivative of the quenched pressure, i.e.,
	\be
	\label{var12-random-quenched}
	\chi(\beta, B) \,  =  \sum_{l \geq 2} p_l^* \frac{\partial^2}{\partial B^2} 
	\log \left(A_+ (\beta, B) +  A_-(\beta, B) \left(r_{\sss B}\right)^l\right) \, 
	+  \frac{\partial^2}{\partial B^2} \log \lambda_+ (\beta, B) .
	\ee
This will allow us to compare the variances in the random and averaged quenched settings.

\subsubsection{Proof of averaged quenched CLT}
\col{Once more},  to prove the CLT \col{in the averaged quenched setting} we must show that
	\be
	\lim_{N \rightarrow \infty}\, \p \left[\exp \left(t \frac{S_{\sss N} - \p\left(S_{\sss N} \right)}{\sqrt{N}}\right) \right] 
	\, = \, \exp\left(\frac{t^2}{2}\sigma_{\mathrm{aq}}^2\right) \quad \quad \quad \mbox{for all}\; t \in \mathbb{R},
	\ee
for some $\sigma_{\mathrm{aq}}^2$. We start by rewriting the moment generating function as a product of two terms that will be analyzed separately, i.e.,
	\be
	\label{two_terms}
	\p \left[\exp \left(t \frac{S_{\sss N} - \p\left(S_{\sss N} \right)}{\sqrt{N}}\right) \right] 
	= \q \left[\frac{Z_{\sss N}\left(\beta, B + \frac{t}{\sqrt{N}}\right)}{Z_{\sss N}(\beta, B)}\right] 
	\, \, \exp\left(-\frac{t}{\sqrt{N}} \p\left(S_{\sss N} \right)\right)\;.
	\ee
From (\ref{tori_linee}), we separate the contribution of lines and tori to the partition function, as
	\be
	\label{e_alla_FeE}
	Z_{\sss N}(\beta, B) = \exp{\left[N F_{\beta, B}\left(p^{\sss(N)}\right) + N E_{\sss N}(\beta, B)\right]},
	\ee
where
	\eqan{
	\label{f-def}
	F_{\beta, B}\left(p^{\sss(N)}\right) 
	&= \log \lambda_+(\beta, B)  
	+ \sum_{l \geq 2} p_l^{\sss{(N)}} \log \left(A_+(\beta, B) + A_- (\beta, B)\left(r_{\sss B}\right)^{l}\right),\\
	\label{e-def}
	E_{\sss N}(\beta, B) &= \frac{1}{N} \sum_{i=1}^{K_{\sss N}^{\sss (t)}} 
	\log \left(1 + \left(r_{\sss B}\right)^{L_{\sss N}^{\sss (t)}(i)} \right),
	}
and $p^{\sss(N)} = \left(p_l^{\sss{(N)}}\right)_{l\geq 2}$ is defined in \eqref{p_elle}. The first factor in \eqref{two_terms} is
	\begin{align}
	\label{qzz}\nonumber
	\q \left[\frac{Z_{\sss N}\left(\beta, B + \frac{t}{\sqrt{N}}\right)}{Z_{\sss N}(\beta, B)}\right]
	&=\q \left[\e^{N\left(F_{\beta, B+\frac{t}{\sqrt{N}}} (p^{\sss(N)}) - F_{\beta, B} (p^{\sss(N)})\right) 
	+ N\left(E_{\sss N} \left(\beta, B+\frac{t}{\sqrt{N}}\right) - E_{\sss N}\left(\beta, B\right)\right) }\right] \\ \nonumber
	&=\q \left[\e^{N\left(F_{\beta, B+\frac{t}{\sqrt{N}}} (p^{\sss(N)}) - F_{\beta, B} (p^{\sss(N)})\right)} 
	\q\left(\left. \e^{N\left(E_{\sss N} \left(\beta, B+\frac{t}{\sqrt{N}}\right) -  E_{\sss N}\left(\beta, B\right)\right)} 
	\right| p^{\sss(N)} \right)\right] \\ 
	&=\q \left[\e^{N\left(F_{\beta, B+\frac{t}{\sqrt{N}}} (p^{\sss(N)}) - F_{\beta, B} (p^{\sss(N)})\right)}\right](1+o(1)),
	\end{align}
because the contribution of $\e^{N\left( E_{\sss N} \left(\beta,B+\frac{t}{\sqrt{N}}\right) - E_{\sss N} \left(\beta, B\right)\right)}$ as $N\to \infty$ is negligible.  Indeed, by Taylor expansion,
	\eqan{
	N\left(E_{\sss N} \left(\beta, B+\frac{t}{\sqrt{N}}\right) - E_{\sss N}(\beta, B)\right)
	&=\frac{t}{\sqrt{N}}\sum_{i=1}^{K_{\sss N}^{\sss (t)}} 
	\frac{L_{\sss N}^{\sss (t)}(i)  \left(r_{\sss B}\right)^{L_{\sss N}^{\sss (t)}(i)-1}\left. \frac{d}{dt}\left( r_{\sss B+\frac{t}{\sqrt{N}}}\right)\right|_{t=0}}{1 + \left(r_{\sss B}\right)^{L_{\sss N}^{\sss (t)}(i)}} 
	+ o\left (\frac{t}{\sqrt{N}} \right)\nonumber\\
	&\leq \frac{t C K_{\sss N}^{\sss (t)}}{\sqrt{N}},
	}
since the summands in the previous equation are uniformly bounded in $N$. Therefore, by \eqref{dueallaK},
	\be
	\q\Big(\e^{N\left(E_{\sss N} \left(\beta, B+\frac{t}{\sqrt{N}}\right) -  E_{\sss N}\left(\beta, B\right)\right)} 
	\Big| p^{\sss(N)} \Big)\leq  \q\left(\left. \e^{\frac{t C K_{\sss N}^{\sss (t)}}{\sqrt{N}}} 
	\right| p^{\sss(N)} \right)\stackrel{N \to \infty}{\longrightarrow} 1,
	\ee
and  then \eqref{qzz} follows.

The next step is to estimate the expectation appearing in  \eqref{qzz}. To this aim, we introduce the notation 
	\be
	\Delta F_{ B+\frac{t}{\sqrt{N}}} (p^{\sss(N)}):=F_{\beta, B+\frac{t}{\sqrt{N}}} (p^{\sss(N)}) - F_{\beta, B} (p^{\sss(N)}),
	\qquad
	f_l (B) \,  := \, \log\left(A_+\left(B\right) + A_-\left(B\right) \left(r_{\sss B}\right)^l\right),
	\ee
where, here and in the rest of the section, the dependence on $\beta$ is dropped for the sake of notation. 
Adding and subtracting  $\sum_{l\geq 2} p_l^* f_l \left(B+\frac{t}{\sqrt{N}}\right)$ to $\Delta F_{B+\frac{t}{\sqrt{N}}} (p^{\sss(N)})$, yields
	\begin{align}
	\Delta F_{B+\frac{t}{\sqrt{N}}} (p^{\sss(N)}) 
	\,= \, &\log \lambda_+\left(B+\frac{t}{\sqrt{N}}\right)-\log \lambda_+(B)
	+\sum_{l\geq 2} \left(p_l^{\sss{(N)}} - p_l^*\right)f_l \left(B+\frac{t}{\sqrt{N}}\right)\nonumber \\
	&+ \sum_{l\geq 2} p_l^* f_l \left(B+\frac{t}{\sqrt{N}}\right) - \sum_{l\geq 2} p_l^{\sss{(N)}} f_l(B). 
	\end{align}
Expanding around $t=0$, we obtain
	\begin{align}
	\Delta  F_{ B+t} (p^{\sss(N)}) 
	= \, &t \frac{\partial}{\partial t} 
	\left(\log \lambda_+\left(B+t\right)\right)\big|_{t=0} 
	\,+\, \frac{t^2}{2}
	\frac{\partial^2}{\partial t^2} 
	\left(\log \lambda_+\left(B+t\right)\right)\big|_{t=0} \\ \nonumber
	&+ t \sum_{l\geq 2} \left(p_l^{\sss{(N)}} - p_l^*\right)
	\frac{\partial}{\partial t} \left(f_l \left(B+t\right)\right)\big|_{t=0}  \\ \nonumber
	&+  \frac{t^2}{2}  \sum_{l\geq 2} \left(p_l^{\sss{(N)}} - p_l^*\right) \left.
	\frac{\partial^2}{\partial t^2} \left(f_l \left(B+t\right)\right)\right |_{t=0} \nonumber \\
	&+ t \sum_{l\geq 2} p_l^*\left.\frac{\partial}{\partial t} \left(f_l \left(B+t\right)\right)\right|_{t=0} 
	+  \frac{t^2}{2}  \sum_{l\geq 2} p_l^*\left.\frac{\partial^2}{\partial t^2} 
	\left(f_l \left(B+t\right)\right)\right|_{t=0} + o(t^2). \nonumber
	\end{align}
Using the above for $t$ replaced with $t/\sqrt{N}$, recalling the expression for the random quenched susceptibility given in \eqref{var12-random-quenched},  taking the exponential of $N\Delta  F_{ B+\frac{t}{\sqrt{N}}} (p^{\sss(N)})$ and the average with respect to $\q$, we get
	\begin{align}\label{prim}
 	&\q \Big[\e^{N(F_{\beta, B+\frac{t}{\sqrt{N}}} (p^{\sss(N)}) - F_{\beta, B} (p^{\sss(N)}))}\Big]\\
 	&\quad=\e^{\frac{t^2}{2}\chi}  \cdot  \q \Bigg[\exp{\Big[t \sqrt{N} 
	\frac{\partial}{\partial t}\log \lambda_+(B+t)\big|_{t=0}
	+t \sqrt{N} \sum_l (p_l^{\sss{(N)}} - p_l^*) 
	\frac{\partial}{\partial t} (f_l (B+t)\big|_{t=0}\Big]}\nn\\
	&\quad\qquad \times\exp{\Big[\frac{t^2}{2}  \sum_l (p_l^{\sss{(N)}} - p_l^*)
	\frac{\partial^2}{\partial t^2} \left(f_l \left(B+t\right)\right)\big|_{t=0} \, 
	+ \, t \sqrt{N} \sum_l p_l^*\frac{\partial}{\partial t} (f_l(B+t))\big|_{t=0}
	\,+ \, o(1)\Big]}\Bigg].\nn
	\end{align}
Next we consider the factor $ \exp\left(-\frac{t}{\sqrt{N}} \p\left(S_{\sss N} \right)\right)$ appearing in \eqref{two_terms}. Since 
	\begin{align}
	\nonumber
 	\p (S_{\sss N}) \, = \, \q \Big[\frac{\partial}{\partial t} \frac{Z_{\sss N} (B+t)}{Z_{\sss N}(B)} \big|_{t=0} \Big],
	\end{align}
we can write 
	\begin{equation}
	\nonumber
	\p(S_{\sss N}) \, = \, N \frac{\partial}{\partial t} \left( \log \lambda_+ \left(B+t\right)\right)\big|_{t=0} 
	\,+\q \Big[\, N \sum_{l \geq 2} p_l^{\sss{(N)}} \frac{\partial}{\partial t} \left(f_l(B+t)\right) \big|_{t=0}  
	+N\frac{\partial  }{\partial t}E_{\sss N}(B+t)\big|_{t=0}\Big].
	\end{equation}
As we \col{have} already done  for $\q \Big[\frac{Z_{\sss N}\left(\beta, B + \frac{t}{\sqrt{N}}\right)}{Z_{\sss N}(\beta, B)}\Big]$, we can show, by bounding the derivative of $E_{\sss N}(\beta , B)$, that the term corresponding to tori gives a vanishing contribution to the limit of $\frac{t}{\sqrt{N}} \p\left(S_{\sss N} \right)$. Thus, we can write
	\begin{equation}\label{sec}
	\e^{- \frac{t}{\sqrt{N}} \p(S_{\sss N})} 
	\, =\, \e^{-t \sqrt{N} \frac{\partial}{\partial t} \left( \log \lambda_+\left(B+t\right)\right)\big|_{t=0} 
	\,-\, t \sqrt{N} \q \Big[\sum_{l \geq 2} p_l^{\sss{(N)}}\frac{\partial}{\partial t} \left(f_l(B+t)\right) \big|_{t=0}\Big]}\Big]
	(1+o(1)).
	\end{equation}
Combining (\ref{prim}) and (\ref{sec}), the moment generating function can be rewritten as
	\be
	\label{formula-in-terms-of-MGF}
	\p \Big[\exp \Big(t \frac{S_{\sss N} - \p\left(S_{\sss N} \right)}{\sqrt{N}}\Big) \Big] 
	= \, \e^{t^2\chi/2} \, \cdot \, 
	\q \Big[\e^{t \sqrt{N} \sum_{l \geq 2} \left(p_l^{\sss{(N)}} - \q(p_l^{\sss{(N)}})\right) 
	\frac{\partial}{\partial t} \left(f_l \left(B+t\right)\right)\big|_{t=0}}\Big](1+o(1)).
	\ee
As the previous equation shows, the cumulant generating function of $ \frac{S_{\sss N} - \p\left(S_{\sss N} \right)}{\sqrt{N}}$  is  expressed  in terms of that of   
	\be
	X_{\sss N}:=\sqrt{N} \sum_{l=2}^{N} \left(p_l^{\sss{(N)}} - \q(p_l^{\sss{(N)}})\right) \frac{\partial}{\partial t} \left(f_l \left(B+t\right)\right)\big|_{t=0},
	\label{defgamma}
	\ee
Since $f_{l}(B)=\log(A_+(B)+A_-(B)(r_{\sss B})^l)=\log(A_+(B))+\log(1+a_{\sss B} (r_{\sss B})^l)$ with $a_{\sss B}=A_-(B)/A_+(B)$, and 
$\log(A_+(B))$ is independent of $l$, so that 
	\eqn{
	\sum_{l=2}^{N} \left(p_l^{\sss{(N)}} - \q(p_l^{\sss{(N)}})\right)\log(A_+(B))\equiv 0,
	}
we can take
	\eqn{
	\label{gamma-l-def}
	\gamma_l(B)\, := \, \frac{\partial}{\partial t}\log(1+a_{\sss B+t} (r_{\sss B+t})^l)\big|_{t=0}.
	}
In particular, $\gamma_l(B)$ is exponentially small for large $l$.\\Thus, in order to complete the proof of the theorem we need to show that the sequence $\left(X_{\sss N}\right)_{\sss N \geq 1}$ converges to a normal limit. This result is provided by the next lemma:

\begin{lemma}[CLT for linear combinations of $(p_l^{\sss{(N)}})_{l\geq 2}$]
\label{averaged_que_CLT}
As $N\to \infty$,
	\be
	\label{XN-rv-def}
	X_{\sss N}\equiv \sqrt{N}\sum_{l=2}^N \left(p_l^{\sss{(N)}} - \q \left(p_l^{\sss{(N)}}\right)\right) \gamma_l(B)
	\, \stackrel{{\cal D}}{\longrightarrow} \, \mathcal{N}\left(0,\sigma_{\sss G}^2\right),
	\ee
with
	\begin{align}
	\label{sigmag2}
	\sigma_{\sss G}^2 \,:= \; &\frac{1-p}{2}\sum_{\ell=2}^{\infty} \gamma_{\ell}^2(B)\left[-\frac{\alpha^{2l-4}}{(1+\alpha)^{2l-2}}
  \left( 1 + \frac{(l-2-\alpha)^2}{\alpha(1+\alpha)} \right) 
  + \frac{1}{1+\alpha} \left(\frac \alpha {1+\alpha}\right)^{l-2}\right]
\nonumber \\
	&+
	\frac{1-p}{2}\sum_{\ell,j=2}^{\infty} \ \gamma_{\ell}(B)\gamma_{j}(B)\left[-\frac{\alpha^{\ell+j-4}}{(1+\alpha)^{\ell+j-2}}
  \left( 1 + \frac{(\ell-2-\alpha)(j-2-\alpha)}{\alpha(1+\alpha)} \right)\right]
,
	\end{align}
where $\alpha=2p/(1-p)$ and $p$ is defined in \eqref{defn1n2}. Also the moment generating function of $X_{\sss N}$ converges to that of $\mathcal{N}\left(0,\sigma_{\sss G}^2\right)$.
\end{lemma}

\noindent
Lemma \ref{averaged_que_CLT} implies that \col{for any $t\in\mathbb{R}$}
	\be
	\lim_{N \rightarrow \infty}\, \q \Big[\e^{t \sqrt{N} \sum_l \left(p_l^{\sss{(N)}} - 
	\q \left(p_l^{\sss{(N)}}\right)\right)\gamma_l(B)}\Big] 
	= \e^{t^2\sigma_{\sss G}^2/2}.
	\ee
This completes the proof of the averaged quenched CLT for $\CMNonetwo$, since
	\be
	\lim_{N \rightarrow \infty}\, \p \Big[\exp \left(t \frac{S_{\sss N} - 
	\p\left(S_{\sss N} \right)}{\sqrt{N}}\right) \Big] \, 
	= \, \exp\left(t^2\sigma_{\mathrm{aq}}^2/2\right),
	\ee
with $\sigma_{\mathrm{aq}}^2 = \chi  +\sigma_{\sss G}^2$.

The remaining part of this section is devoted to the proof of  Lemma \ref{averaged_que_CLT}. This proof is based on the following theorem by de Panafieu and Broutin \cite{dPB} that states a joint CLT for the number of connected components in graphs with degrees 1 and 2:

\begin{theorem}[CLT for connected components  in  random graphs \protect{\cite[Theorem 1]{dPB}}]
\label{cltPB}
Let the random vector $C_{\tau}$ denote the number of connected components of size $\tau$ in a random graph with $n_{1}$ vertices of degree 1 and $n_{2}$ vertices of degree 2 and no other degrees allowed. For some $T\in \mathbb{N}$ and all $2\le \tau \le T$,  let
	\be\label{Cstar}
	C_{\tau}^{*}=\frac{1}{\sqrt{n_{1}/2}}\left (C_{\tau} - \frac{\alpha^{\tau -2}}{(1+\alpha)^{\tau-1}}\frac{n_{1}}{2}\right )
	\ee
where $\alpha=2n_{2}/n_{1}$ is a fixed real positive value, and  $\overline{C^*}^{\sss (N)}=\left (C^{*}_{2},\ldots, C^{*}_{\sss T}\right)$. Then, as $N=n_{1}+n_{2}\rightarrow \infty$, and with $\phi_{\overline{C^*}^{\sss (N)}} (s_{2},\ldots, s_{\sss T})$ denoting the joint moment generating function of $\overline{C^*}^{\sss (N)}$,
	\be
	\lim_{n_{1}\to \infty} \phi_{\overline{C^*}^{\sss (N)}} (s_{2},\ldots, s_{\sss T})=\e^{\bar{s}H\bar{s}^T/2},
	\ee
where $\bar{s}=(s_{2},\ldots, s_{\sss T})$, $\bar{s}^T$ is its transposed, and $H=H(\alpha)$ is the $(T-1)\times (T-1)$-sized covariance matrix given by
	\be
	\label{covmdPB}
	H_{r,t}(\alpha)= -\frac{\alpha^{r+t-4}}{(1+\alpha)^{r+t-2}}
  \left( 1 + \frac{(r-2-\alpha)(t-2-\alpha)}{\alpha(1+\alpha)} \right) 
  + \frac{\indic{r=t}}{1+\alpha} \left(\frac \alpha {1+\alpha}\right)^{r-2},
	\ee
and the convergence is uniform for $\bar{s}$ in a neighborhood of $\bar{0}=(0,\ldots,0)$. As a result, $\overline{C^*}^{\sss (N)}\convd\mathcal{N}(\bar{0},H(\alpha))$, the multivariate normal distribution with covariance matrix $H(\alpha)$, $\alpha=2p/(1-p)$. 
\end{theorem}
Theorem \ref{cltPB}  cannot be applied directly to our case since it deals with the number of {\em all} the connected components of the random graphs, while we are only interested  in the  {\em lines} of  $\CMNonetwo$. 
Therefore, some work is needed in order to adapt Theorem \ref{cltPB} to our setting. To deduce this result, let us introduce some notation. Denote by $\Lambda_{\ell} $ the number of lines of length $\ell$ in $\CMNonetwo$ (omitting the dependence
of $\Lambda_{\ell} $ on $N$ to lighten the notation) and by $\Lambda_{\ell}^{\ast}$ the centered variable normalized by $\sqrt{n_{1}/2}$ as in \eqref{Cstar}.  Then, we have the following Lemma.

\begin{lemma}[CLT for number of lines of fixed lengths in $\CMNonetwo$]
\label{cltLambda}
Let $\Lambda_{\ell}$ be the number of lines of length $\ell$ in a  $\CMNonetwo$ graph with $n_{1}$ vertices of degree one and $n_{2}$ of degree two, as in \eqref{defn1n2}.  For some $T\in \mathbb{N}$ and all $2\le \ell \le T$, let
	\be
	\label{Lambdastar}
	\Lambda_{\ell}^{\ast}=\frac{1}{\sqrt{n_{1}/2}}\left ( \Lambda_{\ell} - \q(\Lambda_{\ell})\right).
	\ee
Then, as $N\to \infty$,  with $\alpha=2n_{2}/n_{1}$  fixed,  the moment generating function of  the vector $\overline{\Lambda^{\ast}}^{\sss (N)}\equiv (\Lambda_{2}^{\ast},\ldots, \Lambda_{\sss T}^{\ast})$ satisfies the following limit
	\eqn{
	\lim_{N\rightarrow \infty}\q\big(\e^{\overline{s}\cdot \overline{\Lambda^{\ast}}^{\sss (N)}}\big )=\e^{\bar{s}H\bar{s}^T/2},
	}
	where $\bar{s}=(s_{2},\ldots, s_{\sss T})$, $\bar{s}^T$ is its transposed, and $H=H(\alpha)$ is the $(T-1)\times (T-1)$-sized covariance matrix given in \eqref{covmdPB}.
As a result, $\overline{\Lambda^{\ast}}^{\sss (N)} \convd \mathcal{N}(\bar{0},H(\alpha))$. 

\end{lemma}

\medskip
\begin{proof}
Consider the number $\Theta_{\ell}$ of tori of length $\ell$  in $\CMNonetwo$  and center and normalize it by  $\sqrt{\frac{n_{1}}{2}}$ to obtain $\Theta_{\ell}^{\ast}$ and let $ \overline{\Theta^{\ast}}^{\sss (N)}=\left (\Theta_{2}^{\ast},\ldots, \Theta_{T}^{\ast} \right)$. 
 Then, $\overline{C^{\ast}}^{\sss (N)}=\overline{\Lambda^{\ast}}^{\sss (N)}+\overline{\Theta^{\ast}}^{\sss (N)}$, where  $\overline{C^{\ast}}^{\sss (N)}$ is defined in \eqref{Cstar}.  In order to compute the limit of the generating function of $\overline{\Lambda^{\ast}}^{\sss (N)}$ 
\be
\q\big(\e^{\overline{s}\cdot \overline{\Lambda^{\ast}}^{\sss (N)}}\big )=\q\big(\e^{\overline{s}\cdot \overline{C^{\ast}}^{\sss (N)} }\e^{-\overline{s}\cdot \overline{\Theta^{\ast}}^{\sss (N)}}\big )
\ee
we follow the same strategy 
that will be used in the proof of Lemma 4.3.  In particular  we use  H\"older's inequality to obtain 
\be\label{ineqholder}
\q\left( \e^{\overline{s}\cdot \overline{C^{\ast}}^{\sss (N)}/p}\right)^{p}/\q\left( \e^{\overline{s}\cdot \overline{\Theta^{\ast}}^{\sss (N)}
q/p}\right)^{p/q}\le \q\big(\e^{\overline{s}\cdot \overline{\Lambda^{\ast}}^{\sss (N)}}\big ) \le \q\left( \e^{\overline{s}\cdot \overline{C^{\ast}}^{\sss (N)}p}\right)^{1/p} \q\left( \e^{-\overline{s}\cdot \overline{\Theta^{\ast}}^{\sss (N)}q}\right)^{1/q}
\ee
where $p=1+\vep$ and $1/p+1/q=1$, so that $q=(1+\vep)/ \vep$ and $\vep>0$. Now, fixed  $s^{*}>0$, we want to prove that 
\be\label{limqteta_star}
\lim_{N\to \infty} \q \left (  \e^{\overline{s}\cdot \overline{\Theta^{\ast}}^{\sss (N)}}\right)  =1
\ee
for all $\bar{s}$ with $\max_i|s_i|=s^{*}$.  Since $\Theta_{\ell}^{\ast}=(\Theta_{\ell}-\q(\Theta_{\ell}))/\sqrt{n_{1}/2}$, we can write
	\be\label{qteta_star}
	\q \left (  \e^{\overline{s}\cdot \overline{\Theta^{\ast}}^{\sss (N)}}\right)  
	= \exp\left 	(-\frac{\sum_{\ell=2}^{T} s_{\ell} \q(\Theta_{\ell})}{\sqrt{n_{1}/{2}}}\right) 
	\q \left (\exp\left (\frac{\sum_{\ell=2}^{T} s_{\ell} \Theta_{\ell}}{\sqrt{n_{1}/{2}}}\right)\right ).
	\ee
Since the number of tori $K_{\sss N}(t)$ can be written as $K_{\sss N}(t)=\sum_{\ell} \Theta_{\ell}$ and when $\max_i |s_i|=s^*$, we have
	\eqn{
	\label{mean-KN(t)-comp}
	\frac{\left  | \sum_{\ell=2}^{T} s_{\ell} \q(\Theta_{\ell}) \right |}{\sqrt{n_{1}/2}} 
	< s^{*}\frac{\q(K_{\sss N}(t))}{\sqrt{n_{1}/2}},
	}
from which, recalling that $\q(K_{\sss N}(t))\sim \log N$ (see \eqref{lim_av_k}), and  $n_{1}\sim \frac{(1-p)}{2} N$, we conclude that the first factor in the right-hand side of \eqref{qteta_star} converges to 1 as $N\to \infty$. In order to deal with the second factor, and again with $\max_i |s_i|=s^*$,
	\be\label{bound_gfteta}
	\q \Big( \exp\Big (-s^{*}\frac{K_{\sss N}(t)}{\sqrt{n_{1}/{2}}}\Big) \Big) 
	\le  \q \Big(\exp\Big(\frac{\sum_{\ell=2}^{T} s_{\ell} \Theta_{\ell}}{\sqrt{n_{1}/{2}}}\Big)\Big)
	\le \q \Big( \exp\Big(s^{*}\frac{K_{\sss N}(t)}{\sqrt{n_{1}/{2}}}\Big) \Big).
	\ee
By convexity of $x\mapsto \e^{x}$, we have that $\expec[\e^{X}]\geq \e^{\expec[X]}$, so the lower bound is at least $1-o(1)$ by \eqref{mean-KN(t)-comp}.

For the upper bound, we note that $K_{\ss N}(t)$ is stochastically dominated by $K_{\sss N}$, the number of tori in a graph of $N$ vertices of degree 2. Thus, recalling that $K_{\sss N}$ is the sum of independent Bernoulli variables, see \eqref{kappa_N} and \eqref{bernI},
	\be
	\q \Big( \exp\Big(s\frac{K_{\sss N}(t)}{\sqrt{n_{1}/{2}}}\Big) \Big)
	\leq\prod_{i=1}^{N}
	\Big( 1+\frac{\e^{s/\sqrt{n_{1}/2}} -1}{2N-2i+1}\Big). 
	\ee
We apply this to $s=s^*>0$ and consider 
	\be\label{bondb}
	0 \le \log \q \Big( \exp\Big(s\frac{K_{\sss N}(t)}{\sqrt{n_{1}/{2}}}\Big) \Big)
	= \sum_{i=1}^{N} \log \Bigg (1+  \frac{ \e^{s/\sqrt{n_{1}/2}} -1 }{2N-2i+1} \Bigg)
	\le \Big(\e^{s/\sqrt{n_{1}/2}} -1\Big)
	\sum_{j=0}^{N-1} \frac{1}{2j+1} 
	\ee
where we use that $\log(1+x)\le x$ for $x>0$.  Since $\sum_{j=0}^{N-1} \frac{1}{2j+1}\sim \log N$ and $\exp\left(s/\sqrt{n_{1}/2}\right) -1 \sim \col{\sqrt{\frac{2}{1-p}}\frac{s}{\sqrt{N}}}$, we conclude that right-hand side of the last display converges to $0$, as required. 
\end{proof} 
\medskip

\begin{proof}[Proof of Lemma \ref{averaged_que_CLT}.] We split $X_{\sss N}$ into two parts as
	\be
	\label{Xk}
 	X_{\sss N}^{\sss (1)}(k) = \, \sqrt{N} \sum_{l=2}^k \left(p_l^{\sss{(N)}} - \q(p_l^{\sss{(N)}})\right) 	\gamma_l(B),
	\qquad
	\, X_{\sss N}^{\sss (2)}(k)=X_{\sss N} -X_{\sss N}^{\sss (1)}(k),
	\ee
where we fix $k\in \{2,\ldots, N-1\}.$ Then, we use H\"older's inequality to bound
	\eqn{
	\label{MGF-ub}
	\q\big(\e^{X_{\sss N}}\big)=\q\big(\e^{X_{\sss N}^{\sss (1)}(k)}\e^{X_{\sss N}^{\sss (2)}(k)}\big)
	\leq	\q\big(\e^{p X_{\sss N}^{\sss (1)}(k)}\big)^{1/p}\q\big(\e^{qX_{\sss N}^{\sss (2)}(k)}\big)^{1/q},
	}
where $p=1+\vep$ and $1/p+1/q=1$, so that $q=(1+\vep)/\vep$. Further, with the same choices of $p,q$, we can also bound
	\eqn{
	\q\big(\e^{X_{\sss N}^{\sss (1)}(k)/p}\big)=\q\big(\e^{X_{\sss N}/p}\e^{-X_{\sss N}^{\sss (2)}(k)/p}\big)
	\leq	\q\big(\e^{X_{\sss N}}\big)^{1/p}\q\big(\e^{X_{\sss N}^{\sss (2)}(k)(q/p)}\big)^{1/q},
	}
so that
	\eqn{
	\label{MGF-lb}
	\q\big(\e^{X_{\sss N}}\big)\geq \q\big(\e^{X_{\sss N}^{\sss (1)}(k)/p}\big)^{p}/\q\big(\e^{X_{\sss N}^{\sss (2)}(k)(q/p)}\big)^{p/q}.
	}
We aim to prove that, for every $b\in \R$ and $k\in \N$,
	\eqn{
	\label{aim-CLT-MGF-1}
	\lim_{N\rightarrow\infty} \q\big(\e^{b X_{\sss N}^{\sss (1)}(k)}\big)=\e^{b^2 \sigma^{2}_{\sss G}(k)/2},
	}
where $\lim_{k\rightarrow \infty} \sigma_{\sss G}(k)^2=\sigma_{\sss G}^2$ (we give the explicit expression of $\sigma^{2}_{\sss G}(k)$ and $\sigma^{2}_{\sss G}$ at the end of the present section) and, for every $b\in \R$,
	\eqn{
	\label{aim-CLT-MGF-2}
	\limsup_{k\rightarrow \infty}
	\limsup_{N\rightarrow\infty} \q\big(\e^{b X_{\sss N}^{\sss (2)}(k)}\big)=1.
	}
Substituting \eqref{aim-CLT-MGF-1}--\eqref{aim-CLT-MGF-2} into \eqref{MGF-ub}--\eqref{MGF-lb} and letting $\vep\searrow 0$ with $p=1+\vep$, $q=(1+\vep)/ \vep$ concludes the proof of Lemma \ref{averaged_que_CLT}.
We continue with the proofs of \eqref{aim-CLT-MGF-1} and \eqref{aim-CLT-MGF-2}.

\paragraph{Proof of \eqref{aim-CLT-MGF-1}.} This is a direct consequence of Lemma \ref{cltLambda}.
\qed

\paragraph{Proof of \eqref{aim-CLT-MGF-2}.} Here we need to show that we can effectively truncate the sum over $k$. We start by effectively going to the single variable case. We denote 
	\eqn{
	\gamma_{\sss >k}=\sum_{\ell>k} |\gamma_{\ell}|,
	\qquad
	q_{\ell}=|\gamma_{\ell}|/\gamma_{\sss >k}.
	}
Then, we note that, with $Y$ the random variable for which $\prob(Y=\ell)=q_{\ell}$ and denoting $b_{\sss N}=b\sqrt{N}$, we can rewrite
	\eqan{
	\q\big(\e^{b X_{\sss N}^{\sss (2)}(k)}\big)
	&=\q\Big(\exp{\Big\{b_{\sss N}\gamma_{\sss >k} 
	\expec_{\sss Y}\big[\sign(\gamma_Y)\left(p_Y^{\sss{(N)}} - \q(p_Y^{\sss{(N)}})\right)\big]\Big\}}\big)\\
	&\leq \q\Big(\expec_{\sss Y}\Big[\exp{\Big\{b_{\sss N}\gamma_{\sss >k} 
	\sign(\gamma_Y)\left(p_Y^{\sss{(N)}} - \q(p_Y^{\sss{(N)}})\right)\Big\}}\Big]\Big)\nn,
	}
where the inequality follows by Jensen's inequality and $\expec_{\sss Y}$ denotes the expectation w.r.t.\ $Y$ (keeping all other randomness fixed). Thus,
	\eqan{
	\q\big(\e^{b X_{\sss N}^{\sss (2)}(k)}\big)
	&\leq \sum_{\ell>k} q_{\ell} \q\Big(\e^{b_{\sss N}\gamma_{\sss >k} 
	\sign(\gamma_\ell)\left(p_\ell^{\sss{(N)}} - \q(p_\ell^{\sss{(N)}})\right)}\Big).
	}
Thus, it suffices to bound the moment generating function of $p_\ell^{\sss{(N)}} - \q(p_\ell^{\sss{(N)}})$ for a single $\ell$.

We continue by studying this moment generating function. Denote $N_{\ell}=Np_\ell^{\sss{(N)}}$, so that
	\eqn{
	\q\Big(\e^{b_{\sss N}\gamma_{\sss >k} 
	\sign(\gamma_\ell)\left(p_\ell^{\sss{(N)}} - \q(p_\ell^{\sss{(N)}})\right)}\Big)
	=\q\Big(\e^{b\gamma_{\sss >k} 
	\sign(\gamma_\ell)\left(N_\ell- \q(N_\ell)\right)/\sqrt{N}}\Big).
	}
For $\theta\in {\mathbb R}$, we let
	\eqn{
	M_{n_1,n_2}^{\sss (\ell)}(\theta)=\q\Big(\e^{\theta \left(N_\ell- \q(N_\ell)\right)}\Big).
	}
We prove by induction on $n_1\geq 2$ that there exists $A,\vep>0$ such that, for all $|\theta|\leq \vep$ and all $n_2\geq 0$,
	\eqn{
	\label{Mn1n2-def}
	M_{n_1,n_2}^{\sss (\ell)}(\theta)\leq \e^{An_1\theta^2}.
	}
For $n_1=0$, $M_{n_1,n_2}^{\sss (\ell)}(\theta)\equiv 1$.
This initiates the induction hypothesis. To advance the induction hypothesis, we note that, with 
	\eqn{
	p_{n_1,n_2}(k)=\col{\q}(L(1)=k),
	}
the moment generating function $M_{n_1,n_2}^{\sss (\ell)}(\theta)$ satisfies the recursion
	\eqan{
	\label{Mn1n2-rec}
	M_{n_1,n_2}^{\sss (\ell)}(\theta)&=\sum_{k\geq 2} M_{n_1-2,n_2-k+2}^{\sss (\ell)}(\theta) p_{n_1,n_2}(k) 
	\big[\indic{k\neq \ell} + \indic{k=\ell}\e^{\theta}\big] \e^{-\theta p_{n_1,n_2}(\ell)}.
	} 
We use the induction hypothesis, which is allowed since the summand in \eqref{Mn1n2-rec} is non-negative, to arrive at
	\eqan{
	\label{Mn1n2-rec-bd}
	M_{n_1,n_2}^{\sss (\ell)}(\theta)&\leq \e^{A(n_1-2)\theta^2} \sum_{k\geq 2} p_{n_1,n_2}(k) 
	\big[\indic{k\neq \ell} + \indic{k=\ell}\e^{\col{\theta}}\big] \e^{-\theta p_{n_1,n_2}(\ell)}\\
	&= \e^{A(n_1-2)\theta^2}
	\big[1+ (\e^{\theta}-1)p_{n_1,n_2}(\ell)\big] \e^{-\theta p_{n_1,n_2}(\ell)}.\nn
	} 
Choose $\vep>0$ sufficiently small, so that $\e^{\theta}\leq 1+\theta+\theta^2$ for all $|\theta|\leq \vep$. Then,
	\eqan{
	\label{Mn1n2-rec-bd-a}
	M_{n_1,n_2}^{\sss (\ell)}(\theta)&\leq \e^{A(n_1-2)\theta^2}
	\big[1+ (\theta+\theta^2)p_{n_1,n_2}(\ell)\big] [1-\theta p_{n_1,n_2}(\ell)+\theta^2 p_{n_1,n_2}(\ell)^2]\\
	&\leq \e^{A(n_1-2)\theta^2}[1+4\theta^2 p_{n_1,n_2}(\ell)^2].\nn
	} 
Choosing $A>0$ sufficiently large, we can advance the induction hypothesis. We conclude that \eqref{Mn1n2-def} holds. Applying \eqref{Mn1n2-def} with $\theta=b\gamma_{\sss >k} \sign(\gamma_\ell)/\sqrt{N}$ leads to
	\eqan{
	\q\big(\e^{b X_{\sss N}^{\sss (2)}(k)}\big)
	&\leq \sum_{\ell>k} q_{\ell} \e^{A(b\gamma_{\sss >k})^2n_1/N}\leq \e^{A(b\gamma_{\sss >k})^2},
	}
since $(q_{\ell})_{\ell>k}$ is a probability measure and $n_1\leq N$.
This completes the proof of \eqref{aim-CLT-MGF-2}, since $\gamma_{\sss >k}$ is exponentially small in $k$.
\qed

\vspace{0.4cm}
\noindent
We conclude this section reporting the explicit computation of   $\sigma^{2}_{\sss G}(k)$,  the limiting variance of  $X_{\sss N}^{\sss (1)}(k)$ appearing in equation  \eqref{aim-CLT-MGF-1}  and \col{of}
$\sigma^{2}_{\sss G}$, the limiting variance of $X_{N}$, see \eqref{sigmag2}.\smallskip 

\noindent
{\bf The limiting variance $\sigma_{\sss G}^{2}$.} We express $X_{\sss N}^{\sss (1)}(k)$ as a linear combination of  $\Lambda_{\ell}=Np^{\sss (N)}_{\ell}$ 
recalling that the variables ${\Lambda^{\ast}_{\ell}}$ in Lemma \ref{cltLambda},
are defined as $ {\Lambda^{\ast}_{\ell}}=\left ( \Lambda_{\ell} - \q(\Lambda_{\ell})\right )/\sqrt{\frac{n_{1}}{2}}$. Substituting in \eqref{Xk}, we obtain 
\be\label{XkLambda}
 X_{\sss N}^{\sss (1)}(k) \, = \, \sqrt{\frac{1-p}{2}}(1+o(1)) \sum_{l=2}^k  \gamma_l(B){\Lambda^{\ast}_{\ell}},
\ee
and notice that the variance
	\be
	\mathrm{Var}_{\q}( X_{\sss N}^{\sss (1)}(k)) \, = \frac{1-p}{2}(1+o(1)) \Bigg(\, \sum_{\ell=2}^{k} 
	\gamma_{\ell}^{2}(B)\mathrm{Var}_{\q}({\Lambda^{\sss (N)}_{\ell}}^{\ast})
	+\sum\limits_{\substack{\ell,j=2 \\ \ell \ne j}}^{k} \gamma_{\ell}(B)\gamma_{j}(B) 
	\mathrm{Cov}_{\q}\left({\Lambda^{\sss (N)}_{\ell}}^{\ast},{\Lambda^{\sss (N)}_{\ell}}^{\ast}\right )\Bigg)
	\ee
has a limit  as $N\to \infty$. Indeed, using the fact that the functions $\gamma_{\ell}(B)$ are independent of $N$,   from Theorem \ref{cltPB},
	\be\label{vargk}
	\sigma^{2}_{\sss G}(k):=
	\lim_{N\to \infty}\mathrm{Var}_{\q}( X_{\sss N}^{\sss (1)}(k)) 
	\, = \, \frac{1-p}{2}\Bigg(\sum_{\ell=2}^{k} \gamma_{\ell}^{2}(B) H_{\ell,\ell}
	+\sum\limits_{\substack{\ell,j=2 \\ \ell \ne j}}^{k} \gamma_{\ell}(B)\gamma_{j}(B) H_{\ell,j}\Bigg),
	\ee
where $H_{\ell,j}$ are the elements of the covariance matrix  explicitly given in \eqref{covmdPB} with $\alpha=\frac{2p}{1-p}$. Because of the exponentially small factor $(\frac{\alpha}{1+\alpha})^{\ell+j}$ in $H_{\ell,j}$, 
and the exponential smallness  of $\gamma_{\ell}(B)$ for large $\ell$, see \eqref{gamma-l-def}, the sums in \eqref{vargk} converge absolutely as $k\to \infty$. Thus, we conclude that the limiting variance $\sigma_{\sss G}^2$ exists and is given by \eqref{sigmag2}.
\end{proof}
\medskip

\noindent
{\small
{\bfseries Acknowledgments.}}
We are grateful to Aernout van Enter, with whom we discussed some of the topics in this work, for useful suggestions.
We thank \'Elie de Panafieu and Nicolas Broutin for sharing their preprint \cite{dPB} prior to publication. We acknowledge financial support from  the Italian Research Funding Agency (MIUR) through FIRB project ``Stochastic processes in interacting particle systems: duality, metastability and their applications'', grant n.\ RBFR10N90W.
The work of RvdH is supported in part by the Netherlands
Organisation for Scientific Research (NWO) through VICI grant 639.033.806 and the Gravitation {\sc Networks} grant 024.002.003.


{\small 

}

\end{document}